\RequirePackage{fix-cm}
\documentclass[smallextended]{svjour3}       
\smartqed  

\usepackage{amsmath}
\usepackage{latexsym}
\usepackage{hyperref}
\usepackage{amssymb}
\usepackage{lipsum}
\usepackage{amsfonts}
\usepackage{graphicx}
\usepackage{epstopdf}
\usepackage{amsopn}
\usepackage{color}
\newcommand{\esp}{\mathbb{E}}

\newcommand{\prob}{\mathbb{P}}
\newcommand{\probn}{\mathbf{P}}
\newcommand{\var}{\mathbb{V}}
\newcommand{\alg}{\mathcal{F}}

\newcommand{\re}{\mathbb{R}}

\DeclareMathOperator*{\gap}{gap}

\DeclareMathOperator*{\diag}{diag}
\newcommand{\vertiii}[1]{{\left\vert\kern-0.4ex #1 
\kern-0.4ex\right\vert}}
\newcommand{\Lpnorm}[1]{\vertiii{\,#1\,}_{p}}

\newcommand{\Lpcnorm}[1]{\vertiii{\,#1\,}_{\mathfrak{p}}}

\newcommand{\dist}{\mathsf{d}}
\newcommand{\diam}{\mathsf{diam}}
\DeclareMathOperator*{\argmin}{argmin}
\newcommand{\cvar}{\mathsf{CV@R}}

\newcommand{\ignore}[1]{}
\newcommand{\cbdg}{{\bf c}_{\sf bdg}}

\newcommand{\stheta}{\vartheta}

\newcommand{\parametersgoodness}{(\sigma_*^2,\sigma^2,\alpha,\rho)}
\newcommand{\parametersgreatness}{(\sigma_*^2,\sigma^2,\alpha,p,\kappa_p)}
\newcommand{\parametersgreatnessq}{(\sigma_*^2,\sigma^2,\alpha,q,\kappa_q)}
\newcommand{\parametersgoodnessplus}{(\sigma_*^2,\sigma^2(\stheta,\delta),\alpha,\rho)}
\newcommand{\parametersgreatnessplus}{(\sigma_*^2,\sigma^2(\stheta,\delta),\alpha,p,\kappa_p)}

\newtheorem{assumption}{Assumption}

\def\bx{\mathbf x}

\begin{document}

\title{Sample average approximation with heavier tails II
}
\subtitle{Localization in stochastic convex optimization and persistence results for the Lasso}


\author{Roberto I. Oliveira         \and
        Philip Thompson 
}


\institute{Roberto I. Oliveira \at
              Instituto de Matem\'atica Pura e Aplicada (IMPA), Rio de Janeiro, RJ, Brazil. \\
              \email{rimfo@impa.br}           
           \and
           Philip Thompson \at
           Purdue University, Krannert School of Management, West Lafayette, USA. \\
           FGV EMAp, School of Applied Mathematics, Rio de Janeiro, Brazil. \\
           \email{philipthomp@gmail.com}
}
\date{Received: date / Accepted: date}

\maketitle

\begin{abstract}{
``Localization'' has proven to be a valuable tool in the Statistical Learning literature as it allows sharp risk bounds in terms of the problem geometry. Localized bounds seem to be much less exploited in the Stochastic Optimization literature. In addition, there is an obvious interest in  both communities in obtaining risk bounds that require weak moment assumptions or ``heavier-tails''. In this work we use a localization toolbox to derive risk bounds in two specific applications. The first is in portfolio risk minimization with conditional value-at-risk constraints. We consider a setting where, among all assets with high returns, there is a portion of dimension $g$, unknown to the investor, that has  significant less risk than the other remaining portion. Our rates for the SAA problem show that ``risk inflation'', caused by a multiplicative factor, affects the statistical rate only via a term proportional to $g$. As the ``normalized risk'' increases, the  contribution in the rate from the extrinsic dimension diminishes while  the dependence on $g$ is kept fixed. Localization is a key tool to show this property. As a second application of our localization toolbox, we obtain sharp oracle inequalities for least-squares estimators with a Lasso-type constraint under weak moment assumptions. One main consequence of these inequalities is to obtain \emph{persistence}, as posed by Greenshtein and Ritov, with covariates having heavier tails. This 
gives improvements in prior work of Bartlett, Mendelson and Neeman.
}
\end{abstract}

\section{Introduction}

One fundamental problem in Stochastic Programming is to understand the behavior of {\em sample average approximations}  \cite{roemisch2003,shapiro:dent:rus}.  Suppose we are given an optimization problem:
\begin{eqnarray}
f^*:=\min_{x\in Y} &&\quad f_0(x)\nonumber\\
\mbox{s.t.}&&\quad f_i(x)\le0,\quad\quad\forall i\in\mathcal{I},\label{problem:min}
\end{eqnarray}
with $Y\subset\re^d$ and a nonempty \emph{feasible set} 
\begin{eqnarray}
X:=\left\{x\in Y:f_i(x)\le0, \forall i\in\mathcal{I}\right\}.
\label{equation:feasible:set}
\end{eqnarray}
Additionally, assume each of the functions $f_i$ is given by an expectation
\begin{align}
f_i(x) = \probn\,F_i(x,\cdot) := \int_{\Xi}\,F_i(x,\xi)\,\probn(d\xi)\label{eq:mean:fi}
\end{align}
where $\probn$ is a probability measure over a set 
$\Xi$ and the $F_i:Y\times \Xi\to \re$ are Carath\'{e}odory functions. 

\paragraph{Sample average approximation.} In typical settings, the measure $\probn$ and the functions $f_i$ are not directly accessible but a {\em random sample} $\{\xi_k\}_{k=1}^N$ from $\probn$ is available. If that is the case, it is natural to consider the sample-average approximation (or SAA) to (\ref{problem:min}), where the $f_i$ are replaced by sample averages:
\begin{align}
\widehat{F}_i(x):= \widehat{\probn}\,F_i(x,\cdot) = \frac{1}{N}\sum_{k=1}^NF_i(x,\xi_k).\label{eq:mean:hat:fi}
\end{align}
The following natural questions have been considered in numerous works in stochastic optimization:
\begin{enumerate}
\item Are (nearly) optimal solutions to the SAA also nearly feasible and nearly optimal for the original problem (\ref{problem:min})?
\item Are the values of the two problems typically close?
\end{enumerate}

One approach to answer Questions 1 and 2 above is \emph{asymptotic} in nature: the sample size $N$ diverges whereas the functions $F_i$, $f_i$, the set $Y$ and the measure $\probn$ remain fix. These type of results are important as they justify consistency and the construction of validation procedures. Numerous works from the optimization literature have undertaken this program, e.g.,  \cite{artstein:wets,dupacova:wets,king:rockafellar,king:wets,%
pflug1995,pflug1999,pflug2003,shapiro1989,shapiro1991,shapiro2003,shapiro:dent:rus}. See \cite{shapiro:dent:rus,tito:bayraksan,pasupathy} for extensive reviews. 

Alternatively, \emph{non-asymptotic} analysis, the main focus of this work, consists in obtaining bounds for the value and quality of SAA solutions with explicit dependence on the sample size $N$ and other problem parameters. For instance, letting $f^*$ and $\widehat{F}^*$ be the values of the original problem (\ref{problem:min}) and its SAA (respectively), a recent non-asymptotic result by Guigues, Juditsky and Nemirovski \cite{guigues:juditsky:nemirovski} gives guarantees of the form:
 \begin{equation}\label{eq:guarantees}\forall t\geq 0\,:\,\prob\left\{|\widehat{F}^* - f^*|\leq \frac{A + B\sqrt{t}}{\sqrt{N}}\right\} \geq 1 - e^{-t},\end{equation}
 where $A$ and $B$ do not depend on $N$ or $t$ (but do depend on other problem parameters). Guarantees of this kind are called ``sub-Gaussian" because they imply that the tail decay $\sqrt{N}|\widehat{F}^* - f^*|$ roughly matches that of a Gaussian distribution with standard deviation $B$. 
 
In a nutshell, the goal of the companion paper \cite{2020oliveira:thompsonI} was to obtain a more precise understanding of the SAA in finite-sample settings, with a focus on high-dimensional problems. In particular, \cite{2020oliveira:thompsonI} explores how the nonasymptotic properties of a SAA depend on the structure of the ideal optimization problem (\ref{problem:min}), including its dimension, shape of feasible set and growth of the objective function. In the context of equation (\ref{eq:guarantees}), this means understanding how the constants $A$ and $B$ actually depend on the SAA. We note that \cite{2020oliveira:thompsonI} focus on problems where the functions $F_i$ are continuous; this excludes chance constraints, which typically require specific techniques.
 
 Some of the technical contributions of \cite{2020oliveira:thompsonI} include making weak moment guarantees on the data (in particular, finite moment assumptions are all we need) and using better ``chaining"~methods in the analysis of the fluctuations of the SAA. Additionally, in the favorable setting where the ideal problem is convex; a Slater-type condition is satisfied; and the objective function satisfies a growth condition, the error bounds of the problem depend only on the local geometry of the feasible set around the optimum point. As explained in \cite{2020oliveira:thompsonI}, this is a manifestation of the same ``localization"~phenomenon that has been explored in the Statistical Learning literature \cite{bartlett2005,bartlett2006,koltchinskii2006,2011koltchinskii}. In particular, there are cases where this phenomenon may lead to faster-than-$(1/\sqrt{N})$ decay of the errors or other favorable properties. Theorem \ref{thm:convexrandomset} in the present paper is a restatement of this result from \cite{2020oliveira:thompsonI}.

\subsection{Our contribution}
 
 Applying the general method from \cite{2020oliveira:thompsonI} to convex problems is not at all simple, and requires detailed calculations in each specific case. For instance, in \cite{2020oliveira:thompsonI}, we apply our general localized bounds to the case when the problem is locally strong convex or weak sharp. See Propositions 3 and 4 in this paper.

 The goal of the present paper is to work out two particular applications in detail. 
 In the first one, presented in \S \ref{sub:intro:CVaR}, we apply the results of \cite{2020oliveira:thompsonI} directly to a problem with linear objective function and a constraint on the conditional value-at-risk. 
 
  The second, more challenging application considers the LASSO estimator in a setting where weak moment  assumptions are made on the data. In this application, direct application of Theorem \ref{thm:convexrandomset} is not sufficient; instead, we resort to the "localization toolbox"~that was developed in \cite{2020oliveira:thompsonI} to obtain a sharper result. Our LASSO result is  discussed in \S \ref{sub:intro:LASSO}.
 
\subsection{A problem with risk constraints}\label{sub:intro:CVaR}

This example is a caricature of a portfolio maximization problem where there is a constraint on the conditioned value-at-risk $\cvar$ of the losses. It turns out that the localization theory in \cite{2020oliveira:thompsonI}, as encapsulated in Theorem \ref{thm:convexrandomset} below, has a crucial impact in this application. 

Here, \[\xi = (\xi[1],\dots,\xi[d_0])^T\] is a random vector whose coordinates correspond to losses of $d_0$ distinct financial assets. If $x=(x[1],\dots,x[d_0])^T$ is a vector whose coordinates describe the fractions of the initial capital invested in assets $1,\ldots,d_0$, then the total loss is proportional to $\langle x,\xi\rangle$. We wish to minimize the expectation of $\langle x,\xi\rangle$ subject to a constraint on the conditional value-at-risk of the solution \cite{rockafellar:urysaev2000,guigues:juditsky:nemirovski}. In this problem, the case of light-tailed $\xi$ would be of little interest. In Section \ref{sec:cvar}, we describe a specific Assumption \ref{assump:saacvar} that allows heavier tails. In principle, the most general problem would associate ``high risk'' to every asset having high returns. Still, it makes sense to envision a case in which, among these assets, there is a portion of them which might have significant more risk than the other portion having, say, $g$ assets. This is the content of Assumption \ref{assump:saacvar} in Section \ref{sec:cvar}. 

Under Assumption \ref{assump:saacvar}, we obtain statistical rates for the SAA problem showing that ``risk inflation'', caused by a multiplicative factor, affects the statistical rate only via a term proportional to $g$. As the ``normalized risk'' increases, the  contribution from the extrinsic dimension $d_0$ in the rate diminishes while  the dependence on $g$ is kept fixed. Localization is the key tool to show this type of property.

\subsection{The LASSO estimator}\label{sub:intro:LASSO}

The ordinary least squares estimator is known to suffer from consistency problems in high dimensions (for instance, it is an underdetermined optimization problem). The LASSO estimator, introduced by Tibshirani \cite{tibshirani1996}, circumvents this difficulty via a $\ell^1$ restriction of the feasible set, or via $\ell^1$ penalization of potential solutions. 

The LASSO has generated a huge amount of practical as well as theoretical interest. One perspective on it is that it promotes sparsity of solutions. Assume that the data for the problem takes the form:
\[\xi=(y,\bx)\in\re\times \re^d;\]
here $y$ should be thought of as a response variable, and $\bx$ as a set of covariates. Assume further that $y$ can be written as the sum of $\langle x_*,x\rangle$ (with $x_*\in\re^p$ unknown) and a mean-$0$ noise term with variance $\sigma^2$. Much work on the LASSO shows results of the following form: if the vector $x_*$ is sparse, then the LASSO solution approximates it well \cite{2007bunea:tsybakov:wegkamp}. Moreover, generally, the LASSO solution  $\widehat{x}$ satisfies an ``oracle inequality with respect to sparse predictors." Concretely, this means that the penalized estimator satisfies inequalities of the following general form \cite{2009bickel:ritov:tsybakov}:
\[\|\mathbf{\widehat\Sigma}^{1/2}(\widehat{x}-x_*)\|\leq C\,\min_{x\in\re^p}\left(\|\mathbf{\widehat\Sigma}^{1/2}(x-x^*)\| + \sigma^2\mu^2(x^*)\log d\frac{\|x\|_0}{N}\right),\]
with $\mathbf{\widehat\Sigma}$ the ``design matrix" determined by sample points and $C\ge1$ a constant. In above, $\mu(x^*)$ is the ``restricted eigenvalue constant'' depending on $x^*$. Results of this form are now numerous, but they are somewhat restricted in application: they usually require light-tailedness of $\xi$, as well as strong assumptions on the design matrix, such as the Restricted Eigenvalue Property introduced in \cite{2009bickel:ritov:tsybakov}. More details about this line of research are given in Remark \ref{rem:oracleLASSO} below.

A separate line of research, which we pursue here, makes no assumptions on the design matrix, and obtains results of a different kind. Greenshtein and Ritov \cite{2004greenshtein:ritov} introduced the notion of \emph{persistence} of an estimator. While their main result is asymptotic, one may summarize it in a nonasymptotic fashion. Given $R>0$, let $\widehat{x}_R$ denote a solution to the optimization problem:
\[\min_{x\in\re^d,\,\|x\|_1\leq R}\,\widehat{\probn}\,(y-\langle \bx,x\rangle)^2.\]
The main problem of \cite{2004greenshtein:ritov} was the following: how large can one take $R$ in terms of other problem parameters so that $\probn\,(y-\langle \bx,\widehat{x}_R\rangle)^2$ -- the ``population risk"~associated with $\widehat{x}_R$ -- is close to the smallest possible risk $\inf_{\|x\|_1\in R}\probn\,(y-\langle\mathbf{x},x\rangle)^2$ that is possible under the same $\ell^1$ norm constraint? Note that this formulation of the problem makes sense even in the absence of a ``ground truth"~$x_*$; in Statistical Learning lingo, one wishes $\widehat{x}_R$ to have nearly the same risk as the best $x\in\re^d$ satisfying $\|x\|_1\leq R$. 

Both \cite{2004greenshtein:ritov} and the subsequent paper \cite{2006greenshtein} obtain results under relatively strong assumptions on the data that we do not reproduce here. 
Bartlett, Mendelson and Neeman  \cite{2012bartlett:mendelson:neeman} significantly improve earlier results by obtaining stronger, nonasymptotic results. In precise terms, their result is as follows (Theorem 4.5 in \cite{2012bartlett:mendelson:neeman}). Suppose $(y,\bx)\in\re\times\re^d$ is such that each coordinate $\bx_j$ of $\bx$ is \emph{subexponential}. Let $h:=\log^{3/2} N\log^{3/2}(N d)$ and set $\rho=\max_{j\in[d]}\Vert \bx_j\Vert_{\psi_1}$, where $\Vert\cdot\Vert_{\psi_1}$ denotes the subexponential Orlicz norm. Then, for some constant $C>0$ and for any $\delta\in(0,1)$, the constrained Lasso estimator $\widehat{x}_R$ satisfies, with probability $\geq 1-\delta$:
\begin{eqnarray}\label{eq:LASSOBMN} & \probn\,(y-\langle \bx,\widehat{x}_R\rangle)^2  -  \inf_{\|x\|_1\in R}\probn\,(y-\langle\bx,x\rangle)^2
\\ \nonumber \leq & \frac{C}{\delta^2}\left\{\sqrt{\inf_{\|x\|_1\in R}\probn\,(y-\langle\bx,x\rangle)^2}\,\frac{Rh\rho}{\sqrt{N}} + \frac{R^2h^2\rho^2}{N}\right\}.\end{eqnarray}
As a corollary, the right hand side (RHS) is typically small whenever $R\ll \sqrt{N}/\log(dN)^c$, where $c>0$ is a positive constant.  The authors also discuss in \cite{2012bartlett:mendelson:neeman} the application of their results to the problem of convex aggregation of least-squares estimators with weights constrained in the $\ell_1$-ball. We note that in this result the ``slow rate'' $1/\sqrt{N}$ can be achieved without extra assumptions. As noted above, the oracle inequalities described above achieve faster $1/N$ rates, at the cost of stronger assumptions on the design matrix. For a more in depth discussion on this topic, we refer to \cite{2012lecue:mendelson} and references therein.

One main contribution of this paper is to obtain an improved consistency result in the line of (\ref{eq:LASSOBMN}). Specifically, we consider a small variant of the LASSO, and prove an improved version of the above inequality, where only finite moment assumptions on the covariates are needed. In fact, even in the subexponential setting, our inequality improves the $\log N$ and $\log d$ factors in \eqref{eq:LASSOBMN}. Most importantly, we also obtain a logarithmic dependency on $1/\delta$. As we will see, our proof strategy is distinct than the one used in \cite{2012bartlett:mendelson:neeman}. The localization arguments in  \cite{2012bartlett:mendelson:neeman}  makes use of the ``isomorphic method'' initiated in \cite{bartlett2006}. In this work, we use the localization toolbox developed in our companion paper \cite{2020oliveira:thompsonI} as a building block of the proof.  See Proposition \ref{lem:convexlocalized} in Section \ref{sec:deviationlocalization}.

\begin{remark}\label{rem:oracleLASSO}Because the literature on the LASSO is so extensive, we give here a few additional pointers to it that are not directly related to this work. Several oracle inequalities for LASSO have been established  \cite{2007bunea:tsybakov:wegkamp,2009bickel:ritov:tsybakov,2011koltchinskii:lounici:tsybakov,2018bellec:lecue:tsybakov}; this includes results where the LASSO is applying for model selection in nonparametric regression or, as an aggregator of estimators. See also \cite{2004bunea:tsybakov:wegkamp,2006bunea:tsybakov:wegkamp,2007bunea:tsybakov:wegkamp-aggregation:gaussian,2007bunea:tsybakov:wegkamp-density,2008geer}. Bickel, Ritov and Tsybakov \cite{2009bickel:ritov:tsybakov} improve on the oracle inequalities of Bunea, Tsybakov and Wegkamp \cite{2007bunea:tsybakov:wegkamp} obtaining fast rates for LASSO under the Restricted Eigenvalue condition, the weakest known to date. The first sharp oracle inequalities for LASSO were obtained in \cite{2011koltchinskii:lounici:tsybakov}. In \cite{2018bellec:lecue:tsybakov}, oracle inequalities for LASSO and Slope regularization are obtained showing that the tuning parameter can be taken independent of the confidence level and achieving the minimax rate (improving on log terms). We also refer to \cite{2007candes:tao,2009koltchinskii-dantzig,2009bickel:ritov:tsybakov} for oracle inequalities of a related estimator, the so called Dantzig selector \cite{2007candes:tao}, which minimizes the 
$\ell_1$-norm subjected to an approximate KKT condition of the empirical squared loss. Oracle inequalities using entropy or $\ell_p$ norms (with $p$ close to 1) were studied in \cite{2009koltchinskii,2009koltchinskii-dantzig,2009koltchinskii-entropy}. Regarding LASSO, see also \cite{2006zhang:huang,%
2009zhang} for oracle inequalities related to the even more challenging problem of \emph{variable selection} \cite{2007meinshausen:yu,%
2007meinshausen,2007meinshausen:buhlmann,%
2006zhao:yu,2006zou,2006leng:lin:wahba,2008lounici}, that is, checking if the ground truth and the estimator have the same zero coordinates.\end{remark}

\subsection{Organization}

The remainder of the paper is organized as follows. Section \ref{section:notation} presents some preliminaries, including notation. Section \ref{sec:basic} recalls some of the assumptions, definitions and results from the companion paper \cite{2020oliveira:thompsonI}. The example of the conditioned-value-at-risk is presented in section \ref{sec:cvar}, and section \ref{subsection:lasso} discusses the LASSO example. An appendix contains proofs of some additional technical results.

\section{Preliminaries}\label{section:notation}

\subsection{Basic notation}  

Given a set $S$, we denote its (potentially infinite) cardinality by $|S|$. The complement of an event $E$ in a probability space is $E^c$. For $m\in\mathbb{N}$, we write $[m]:=\{1,\ldots,m\}$. 

Elements of $\re^d$ are column vectors. Given $x\in\re^d$, its coordinates are denoted by $x[i]$, $1\leq i\leq d$. A superscript $T$ is used to denote transposition of a vector, so $x\in\re^d$ is given by $(x[1],\dots,x[d])^T$. The inner product of $x,y\in\re^d$ is denoted by $\langle x,y\rangle$ or $x^Ty$. Norms are denoted by $\|\cdot\|$ and the unit ball around $0$ in that norm is $\mathbb{B}$. Given $a\in\re$, $a_+:=\max\{a,0\}$. Given positive numbers $x$ and $y$, $x\lesssim y$ means that $x\le Cy$ for an absolute constant $C>0$.

Let $(\mathcal{M},\dist)$ be a metric space. We let $\diam(A)$ denote the (potentially infinite) diameter of $A\subset \mathcal{M}$. Given $x\in \mathcal{M}$ and $A\subset \mathcal{M}$ nonempty, $\dist(x,A):=\inf_{a\in \mathcal{A}}\dist(x,a)$.

We fix from now on a probability space $(\Omega,\cal{A},\prob)$ and assume all random variables we consider are defined on it. Given a random variable $Z$, we let $\esp[Z]$ denote its mean, $\var[Z]$ denote its variance and $\Lpnorm{Z}:= (\esp[|Z|^p])^{1/p}$ denotes $L^p$ norm (for $p\geq 1$). 

\subsection{Complexity parameters for sets} 

We review in this section some definitions and results about ``generic chaining". Talagrand's book \cite{talagrand2014} is the best reference for these concepts. 

The ``generic chaining"~functional of a metric space $(\mathcal{M},\dist)$ is a measure of the ``complexity"~of discretizing $\mathcal{M}$ at different scales. To define it, we need the following concept. A sequence $\{\mathcal{A}_j\}_{j=0}^{+\infty}$ is {\em admissible} if each $\mathcal{A}_j$ is a partition of $\mathcal{M}$, with $|\mathcal{A}_0|=1$ and $|\mathcal{A}_j|\leq 2^{2^j}$ for each $j\geq 1$. For each $j$, we let $\diam(\mathcal{A}_j)$ to denote the largest diameter of a set in partition $\mathcal{A}_j$. Talagrand's $\gamma_2$-functional is then defined by
\begin{equation}\label{eq:defgamma2}\gamma_{2}^{(\alpha)}(\mathcal{M},\dist):=\inf\limits_{\{\mathcal{A}_j\}_{j}\text{ admissible}}\left\{\,\sum_{j\geq 0}2^{\frac{j}{2}}\diam(\mathcal{A}_j)^{\alpha}\right\}.
\end{equation}

Talagrand's celebrated majorizing measures theorem \cite{talagrand1994,talagrand2014} shows that:
\begin{equation}\label{eq:talagrand}c\,\esp\left[\sup_{x\in \mathcal{M}}|Y_x-Y_{x_0}|\right]\leq \gamma_{2}^{(\alpha)}(\mathcal{M},\dist)\leq C\,\esp\left[\sup_{x\in \mathcal{M}}|Y_x-Y_{x_0}|\right],\end{equation}
with $c,C>0$ universal, when the $Y_x$ are mean-zero Gaussian and $\var[Y_x-Y_{x'}] =\dist(x,x')^{2\alpha}$. 

The functional $\gamma^{(\alpha)}_2(\mathcal{M},\dist)$ is somewhat mysterious, and can be quite difficult to compute. One can upper bound it via Dudley's {\em entropy integral} \cite{talagrand2014}. Recall that an $r$-net in $\mathcal{M}$ is a set $A\subset \mathcal{M}$ such that $\dist(x,A)\leq r$ for all $x\in \mathcal{M}$. The $r$-covering number of $\mathcal{M}$ is the size of the smallest $r$-net. The {\em $r$-entropy number of $\mathcal{M}$}, $\mathsf{H}(\mathcal{M},r)$, is the natural log of the $r$-covering number. It is known that
\begin{equation}\label{eq:entropybound}\gamma^{(\alpha)}_2(\mathcal{M})\leq C\,\int_0^{\diam(\mathcal{M})}\,\sqrt{\mathsf{H}(\mathcal{M},r^{\frac{1}{\alpha}})}\,dr,\end{equation}
with $C>0$ is a universal constant. An important special case is when $\mathcal{M}\subset \re^d$ and $\dist$ is given by a norm, in which case the entropy integral bound is upper bounded by $\diam(\mathcal{M})^{\alpha}\sqrt{d}$ up to a universal constant. In particular, we obtain, 
\begin{equation}\label{eq:simpleboundgamma2}\gamma_{2}^{(\alpha)}(\mathcal{M})\leq C_{\alpha}\sqrt{d}\,\diam(\mathcal{M})^{\alpha}\end{equation}
with $C_{\alpha}>0$ only depends on $\alpha$. However, this bound can be very loose, as the next example shows. 

\begin{example}Let $\mathcal{M}\subset \re^d$ denote the standard simplex in $d$ dimensions, that is, the convex hull of the $d$ canonical basis vectors. Let $\dist$ denote the standard Euclidean metric. In this case, (\ref{eq:simpleboundgamma2}) bounds $\gamma_{2}^{(1)}(\mathcal{M}) = \mathcal{O}(\sqrt{d})$. By contrast, (\ref{eq:talagrand}) shows that $\gamma_{2}^{(\alpha)}(\mathcal{M})$ is of the order of $\sqrt{\log d}$.\end{example}

Finally, we state a proposition that will be useful later. It shows that the generic chaining functional is well-behaved under Cartesian products.  

\begin{proposition}[Proof in the Appendix]\label{prop:sizeproduct}Let $(\mathcal{M}_i,\dist_i)$ be metric spaces for $i=1,2$. Set $\mathcal{M}=\mathcal{M}_1\times \mathcal{M}_2$. Define the metric $\dist$ on $\mathcal{M}$ given as follows: for every 
$(x_1,x_2),(y_1,y_2)\in \mathcal{M}$,
$$
\dist((x_1,x_2),(y_1,y_2)):= \dist_{1}(x_1,y_1) + \dist_{2}(x_2,y_2).
$$
Then 
$
\gamma_2^{(\alpha)}(\mathcal{M},\dist)\leq (\sqrt{2}+1)\,(\gamma_2^{(\alpha)}(\mathcal{M}_1,\dist_1) + \gamma_2^{(\alpha)}(\mathcal{M}_2,\dist_2)).
$
\end{proposition}

\section{Setup and results from the companion paper}\label{sec:basic}

In this section we review the notation, assumptions and results from \cite{2020oliveira:thompsonI} that we need in the present paper.

\subsection{Ideal optimization versus SAA}\label{sub:basic}

\paragraph{Functions and sets.} As in the introduction, $\mathcal{I}$ is a finite set which will index the constraints of our problem. We use $0\not\in\mathcal{I}$ to index the objective function and set $\mathcal{I}_0:=\mathcal{I}\cup \{0\}$. 

We are given a set $Y\subset \re^d$ and functions $f_i:Y \to \re$, for $i\in \mathcal{I}_0$. We will also write $f:=f_0$. Given $\delta\in\re$, we define:
\[X_\delta:= \{x\in Y\,:\, \forall i\in\mathcal{I},\, f_i(x)\leq \delta\}.\]
We also write $X$ instead of $X_0$. Note that $X=X_{\delta}=Y$ for all $\delta>0$ when $\mathcal{I}=\emptyset$. The ``ideal"~optimization problem we consider is: 

\begin{eqnarray}
f^*:=\min_{x\in Y}&\quad & f(x) \label{problem:ideal}\\
\mbox{s.t.}&\quad &f_i(x)\leq 0,\, i\in\mathcal{I}.\nonumber
\end{eqnarray} 
In other words, the feasible set is $X$, the objective function is $f=f_0$ and the value of the problem is $f^*$. We will always assume implicitly that $X\neq \emptyset$. We let 
\[x^*\in{\rm arg\, min}_{x\in X}f(x)\mbox{ so that }f^*:=f(x^*).\]
In particular, we assume implicitly that our problem always has minimizers. We also use the symbols:
\[f^*_\delta:=\inf_{x\in X_\delta}f(x)\mbox{ and }{\rm gap}(\delta):=|f^*_\delta-f^*| \mbox{ (when $X_{\delta}\neq \emptyset$)}.\]
In case the above infimum is attained, we let 
\[x^*_{\delta}\in{\rm arg\, min}_{x\in X_\delta}f(x)\mbox{ so that }f^*_\delta:=f(x^*_\delta).\]

We will need some additional notation. We write $X_{\delta,{\rm act}(i)}$ for the subset of $X_\delta$ where constraint $i$ is active:
\[X_{\delta,{\rm act}(i)}:= \{x\in X_\delta\,:\, f_i(x) = \delta\}.\] We also define the set of points $x\in X_\delta$ that achieve $f(x)\leq f^*+\stheta$:
\[X_{\delta}^{*,\stheta}:= \{x\in X_\delta\,:\, f(x)\leq f^* + \stheta\}.\]
We set
\[X_{\delta}^{*,=\stheta}:= \{x\in X_\delta\,:\, f(x) =f^* + \stheta\}\]
and finally
\[X_{\delta,{\rm act}(i)}^{*,\stheta} := X^{*,\stheta}_\delta\cap X_{\delta,{\rm act(i)}}.\]
We emphasize that $\delta$ will be omitted from our notation when it is equal to zero.

\paragraph{Randomness.} Let $(\Xi,\sigma(\Xi),\probn)$ denote a probability space. We write $\xi\sim \probn$ to denote a random element of $\Xi$ with law $\probn$.  In this paper,$\{\xi_k\}_{k=1}^N\subset \Xi$ is an i.i.d. random sample of size $N$ from the probability measure $\probn$. The $\xi_k$ are defined over a common probability space $(\Omega,\mathcal{A},\prob)$ that will be always kept implicit. $\widehat{\probn}$ denotes the empirical measure of the sample:
\[\widehat{\probn}:= \frac{1}{N}\sum_{k=1}^{N}\delta_{\xi_k}.\]
Given a measurable function $H:Y\times \Xi\to \re$ and $y$, we define:
\[\probn H(y,\cdot):= \int_{\Xi}H(y,\xi)\,\probn(d\xi)\mbox{ and }\widehat{\probn}\,H(y,\cdot) := \frac{1}{N}\sum_{k=1}^{N}H(y,\xi_k)\]
to denote the expectation and sample average (respectively) of $H(y,\cdot)$ with $y$ fixed. Our assumptions will be such that the integral over $\probn$ will always be well defined. 

\paragraph{Sample average approximation.} We are given measurable functions $F_i:Y \times \Xi\to \re$ for $i\in \mathcal{I}_0$. We assume that 
\begin{equation}\label{eq:SAAcondition} \forall i\in \mathcal{I}_0\,\forall x\in Y\,:\, \probn|F_i(x,\cdot)|<+\infty\mbox{ and }\probn F_i(x,\cdot) = f_i(x).\end{equation}
Write:\[\widehat{F}_i(x):=\widehat{\probn}F_i(x,\cdot) = \frac{1}{N}\sum_{k=1}^NF_i(x,\xi_k)\]
to denote the sample average of $F_i$. Formally, $\widehat{F}_i(x)$ is a function of $x$ and the sample, but we omit the sample from our notation. We sometimes write $\widehat{F}:=\widehat{F}_0$. The sample average approximation to problem (\ref{problem:ideal}) is:
\begin{eqnarray}
\widehat{F}^*:=\min_{x\in Y}&\quad & \widehat{F}(x) \label{problem:SAA}\\
\mbox{s.t.}&\quad &\widehat{F}_i(x)\leq \hat\delta,\, i\in\mathcal{I}.\nonumber
\end{eqnarray} 
for some $\hat\delta\in\re$. The most typical situation is 
$\hat\delta=0$ but we consider the general case for completeness. 
 
Intuitively, the $\widehat{F}_i$ should give random approximations to the $f_i$ for large $N$, and optimization problems with the 
$\widehat{F}_i$ should be similar to the ``ideal" ~problems involving the $f_i$. We will need the following analogues of the notation introduced above:
\begin{eqnarray*}\widehat{X}&:=& \{x\in Y\,:\,\forall i\in\mathcal{I},\, \widehat{F}_i(x)\leq \hat\delta\};\\
\widehat{X}^{*,\stheta}&:=& \{x\in \widehat{X}\,:\, \widehat{F}(x)\leq \widehat{F}^* + \stheta\}.\\
\widehat{x}^* & \in & {\rm arg min}_{x\in \widehat X} \widehat{F}(x).\end{eqnarray*}

Again, we implicitly assume that the SAA always has solutions. 

\subsection{Assumptions on the random functions} We now give our {\em probabilistic assumptions} on the random functions $F_i$. We start with a definition.

\begin{definition}[Good and great random variables] Given $(\sigma^2,\rho)\in\re_+$, a function $h:\Xi\to \re_+$ is said to be $(\sigma^2,\rho)$-good if $\probn\,h(\cdot)\leq \sigma^2$ and
\[\prob\left\{\widehat{\probn}h(\cdot)> 2\sigma^2\right\}\leq \rho.\]
Given $\sigma^2>0$, $p\geq 2$ and 
$\kappa_p>0$, we say that $h$ is $(\sigma^2,p,\kappa_p)$-great if $\probn h(\cdot)\leq \sigma^2$ and in addition we have the $L^p$ norm bound:
\[\Lpnorm{h(\xi) - \probn h(\cdot)}\leq \kappa_p\sigma^2.\]\end{definition}
Any {\em fixed} integrable function $h\geq 0$ with $\probn h(\cdot)\leq \sigma^2$ is $(\sigma^2,\rho)$-good when $N$ is large enough due to the Law of Large Numbers. The point of our definition is to have finite-$N$ results. The next proposition says that great random variables satisfy a quantitative form of goodness.

\begin{proposition}[Proof in companion paper \cite{2020oliveira:thompsonI}] \label{prop:goodisgreath}If $h$ as above is 
$(\sigma^2,p,\kappa_p)$-great, it is also $(\sigma^2,\rho)$-good, with
\[\rho:=\left(\cbdg\kappa_p\sqrt{\frac{p}{N}}\right)^p\]
and $\cbdg$ is a universal constant.\end{proposition}

The kind of assumption we will make on the $F_i$ is described below. In what follows, $Z\subset Y$ is a subset of $Y$ containing $x^*$, $\sigma^2,\sigma_*^2,\rho>0$, $\alpha\in (0,1]$, $\kappa_p\geq 1$ and $p\geq 2$. Also, $\Vert\cdot\Vert$ is a norm over $\re^d$.

\begin{assumption}[$\parametersgoodness$-goodness over $Z$]\label{assump:goodFi} The functions $\{f_i\}_{i\in\mathcal{I}_0}$ and $ \{F_i\}_{i\in\mathcal{I}_0}$ are continuous in $x\in Y$. Moreover, 
\begin{enumerate}
\item The maps $\xi\mapsto (F_i(x^*,\xi) - f_i(x^*))^2\in \re_+$ are $(\sigma_*^2,\rho)$-good for each $i\in\mathcal{I}_0$;
\item For each map $F_i$ with $i\in\mathcal{I}_0$, there exists $\mathsf{L}_i:\Xi\to\re$ such that $\mathsf{L}_i^2$ is $(\sigma^2,\rho)$-good and:
\[\forall x,x'\in Z,\,\forall \xi \in\Xi\,:\, |F_i(x,\xi) - F_i(x',\xi)|\leq \mathsf{L}_i(\xi)\,\|x-x'\|^\alpha.\]\end{enumerate}
\end{assumption}

\begin{assumption}[$\parametersgreatness$-greatness over $Z$]\label{assump:greatFi} The functions $\{f_i\}_{i\in\mathcal{I}_0}$ and $ \{F_i\}_{i\in\mathcal{I}_0}$ are continuous in $x\in Y$. Moreover, 
\begin{enumerate}
\item The maps $\xi\mapsto (F_i(x^*,\xi) - f_i(x^*))^2\in \re_+$ are $(\sigma_*^2,p,\kappa_p)$-great for each $i\in\mathcal{I}_0$;
\item For each map $F_i$ with $i\in\mathcal{I}_0$, there exists $\mathsf{L}_i:\Xi\to\re$ such that $\mathsf{L}_i^2$ is $(\sigma^2,p,\kappa_p)$-great and:
\[\forall x,x'\in Z,\,\forall \xi \in\Xi\,:\, |F_i(x,\xi) - F_i(x',\xi)|\leq \mathsf{L}_i(\xi)\,\|x-x'\|^\alpha.\] \end{enumerate}
\end{assumption}

In general,  one may require $Z=Y$. In convex settings, we may take potentially much smaller sets $X^{*,\stheta}_{\delta}\subset X$. Notice that each of the above assumptions implies:
\begin{equation}\label{eq:holdercont}\forall i\in\mathcal{I}_0,\,\forall x,x'\in Z\,:\, |f_i(x) - f_i(x')|\leq (\probn\mathsf{L}(\cdot))\,\|x-x'\|^{\alpha}\leq \sigma \,\|x-x'\|^{\alpha},\end{equation}
that is, the functions $f_i$ are $\alpha$-H\"{o}lder continuous over $Z$. In particular, when $Z=Y$, ou assumptions impose restrictions on the growth of the objective function $f_0$.

We note the following simple consequence of Proposition \ref{prop:goodisgreath}.

\begin{proposition}[Great implies good; Proof in companion paper \cite{2020oliveira:thompsonI}] Assumption \ref{assump:greatFi} implies Assumption \ref{assump:goodFi} with the same set $Z$, the same parameters $\sigma^2,\sigma_*^2$, and 
\[\rho:=\left(\cbdg\kappa_p\sqrt{\frac{p}{N}}\right)^p.\]\end{proposition}

In particular, our assumptions may be satisfied with $\rho$ polynomially small in $N$, even if the random variables involved do not have light tails. 

\subsection{Assumptions on the geometry of the problem}\label{sub:geometry}

When there are constraints in expectation, the SAA will unavoidably have a different feasible set than the ideal problem. In this section, we present a standard assumption that allows us to bound the difference between the two sets. It is  standard in the analysis of perturbations and algorithms for problems in Optimization and Variational Analysis \cite{pang,iusem:jofre:thompson2015}. For convex problems, we consider a localized version of the Slater CQ condition. Here, we only require that the set $X^{*,\stheta}$ be bounded and has an ``interior point".

\begin{assumption}[Localized Slater CQ with convexity (LSCQ)]\label{assump:LSCQ} The set $Y$ is convex and closed, and the functions $\{f_i\}_{i\in\mathcal{I}_0}$ are continuous and convex. Moreover, there exist $\eta_*>0$ and $\stheta_*$ such that $X^{*,\stheta_*}$ is bounded and $X_{-\eta_*}^{*,\stheta_*}\neq \emptyset$ (that is, there exists $x\in Y$ with $f(x)\leq f^*+\stheta_*$ and $f_i(x)\leq -\eta_*$ for all $i\in\mathcal{I}$).
\end{assumption}

Boundedness of $X^{*,\stheta_*}$ may be guaranteed by usual assumptions. In that case, the Slater CQ, i.e., $X_{-\eta_*}\neq\emptyset$ for some 
$\eta_*>0$, implies Assumption \ref{assump:LSCQ} with $\stheta_* \ge \inf_{x\in X_{-\eta_*}}f(x)-f^* = {\rm gap}(-\eta_*)$. Assumption \ref{assump:LSCQ} allows us to control the complexity of $X^{*,\stheta}_\delta$ in terms of $X^{*,\stheta}$, for suitable $\stheta$ and $\delta$; see Lemma 5 in \cite{2020oliveira:thompsonI} for details.

We conclude this section by noting that in the next Section \ref{ss:SAA:convex:loc}, the functional $\gamma_2^{(\alpha)}$ is defined with respect to $\dist$, i.e., the set-to-point distance associated to the norm $\Vert\cdot\Vert$ over $\re^d$. See Assumptions \ref{assump:goodFi} and \ref{assump:greatFi}. 

\subsection{Deviation and localization arguments under convexity}\label{sec:deviationlocalization}

This section presents a general deterministic perturbation result which is useful in analyzing how the SAA problem differs from the ideal one. In Proposition \ref{lem:convexlocalized} below, lower and upper tails are distinguished. For \emph{estimation}, the fact that lower bounds hold under much weaker tail assumptions were already exploited in the literature \cite{lecue:mendelson2017,oliveira2013,oliveira2016}. Our \emph{persistency} results for Lasso in Section \ref{subsection:lasso} also makes use of this fact. Proposition \ref{lem:convexlocalized} below is one fundamental ingredient.

\begin{remark}
The results in this section are purely deterministic in the sense that we do not need \eqref{eq:mean:fi} and \eqref{eq:mean:hat:fi} to hold. We simply need to assume that $Y$ is as given in Section \ref{sec:basic}; that $f_i,\widehat{F}_i:Y\to \re$ are functions (with $i\in\mathcal{I}_0$) and that $X$, $\widehat{X}$, etc are defined in terms of the $f_i$ and $\widehat{F}_i$ as prescribed in Section \ref{sec:basic}. Our results will be the most interesting when the $\widehat{F}_i$ are good approximations to the respective $f_i$. Proposition \ref{lem:convexlocalized} below may be useful in other applications where the perturbation of convex optimization problems is of interest. 
\end{remark}

For convenience, we introduce the following notation. Given $x,y\in Y$, $Z\subset Y$ and $i\in\mathcal{I}_0$:
\begin{eqnarray}\label{eq:defDeltaone}\widehat{\Delta}_i(x)&:=&\widehat{F}_i(x) - f_i(x);\\ 
\label{eq:defdeltatwo}\widehat{\Delta}_i(y;x)&:=& \widehat{\Delta}_i(y)  - \widehat{\Delta}_i(x).\end{eqnarray}

\begin{proposition}[Proof in companion paper \cite{2020oliveira:thompsonI}]\label{lem:convexlocalized}Assume that $Y$ is convex and closed and that the functions $\{f_i\}_{i\in\mathcal{I}_0}$ and $\{\widehat{F}_i\}_{i\in\mathcal{I}_0}$ are all convex and continuous. Given 
${\delta^{\circ}},\delta>0$, assume $x_{-\delta^{\circ}}\in X_{-\delta^{\circ}}$. Fix $\epsilon\geq f(x_{-\delta^{\circ}})-f^*$ and $\epsilon_0>0$. Define 
$
\hat\epsilon:=\widehat{F}(x_{-\delta^\circ}) - \widehat{F}^*+\epsilon_0.
$ 
If the following three conditions hold:
\begin{eqnarray}\label{eq:controlconstraintsinterior} 
\max_{i\in\mathcal{I}}\widehat{\Delta}_i(x_{-\delta^{\circ}})&\leq &\hat\delta+\delta^{\circ};\\ 
\label{eq:controlconstraintsboundary}
\min_{i\in\mathcal{I}}\inf_{x\in X^{*,\epsilon}_{\delta,{\rm act}(i)}\cap \widehat{X}^{*,\hat\epsilon}} \widehat{\Delta}_i(x)&> & \hat\delta-\delta;\\
\label{eq:controlobjective} \inf_{x\in X^{*,=\epsilon}_{\delta}\cap \widehat{X}^{*,\hat\epsilon}} \widehat{\Delta}_0(x;x_{-\delta^{\circ}}) &>& -(\epsilon - (f(x_{-\delta^{\circ}}) - f^*)-\epsilon_0),
\end{eqnarray}
then:   
\begin{enumerate}
\item \(\widehat{X}^{*,\epsilon_0}\subset X^{*,\epsilon}_\delta,\) or equivalently, any $x\in\widehat{X}$ with $\widehat{F}(x)\leq \widehat{F}^* + \epsilon_0$ satisfies $x\in Y$, $\max_{i\in\mathcal{I}}f_i(x)\leq \delta$ and $f(x)\leq f^*+\epsilon$;
\item The values of the SAA and the ideal problem satisfy \[|\widehat{F}^*- f^*|\leq |\widehat{\Delta}_0(x^*)| + \sup_{x\in X^{*,\epsilon}_{\delta}}|\widehat{\Delta}_0(x;x^*)| + \max\{\epsilon,\gap(\delta)\}.\]
\end{enumerate}\end{proposition} 

\begin{remark}
In Proposition 5 in \cite{2020oliveira:thompsonI}, we set $\hat\delta=0$. The change in the proof is minor.
\end{remark}

\subsection{Localized bounds for SAA under convexity and heavier tails}\label{ss:SAA:convex:loc}

We now recall a fundamental result from a companion paper \cite{2020oliveira:thompsonI}. Theorem \ref{thm:convexrandomset} below establishes bounds for the SAA problem when the set $Y$, the functions $F_i$ are convex and the feasible set satisfies a localized Slater-type condition (Assumption \ref{assump:LSCQ}). There are two fundamental features in Theorem \ref{thm:convexrandomset} to be noted. The first concerns ``localized'' bounds in terms of the geometry around approximate solutions. Unfortunately, the statement of such precise bounds are somewhat involved. Localized bounds in Statistical Learning are also somewhat technical \cite{bartlett2005}; in our case an additional difficulty lies in the fact that feasibility affects optimality. We recommend reading the preliminary discussion in Section 5.1 in \cite{2020oliveira:thompsonI}. Additionally, in Section 5.3 in \cite{2020oliveira:thompsonI}, we present two concrete applications of Theorem \ref{thm:convexrandomset} when the problem has a local regular solution set. They include locally strongly-convex and locally weak sharp problems. In Section \ref{sec:cvar},  we apply Theorem \ref{thm:convexrandomset} to the risk-averse optimization problem, clarifying how to use it in a concrete application. A second observation is that the bounds stated in Theorem \ref{thm:convexrandomset} hold for a large class of heavier tails. Precisely, it distinguishes two separate components: a subgaussian tail dictated by the geometry of near-optimal solutions and the heavy-tailed behavior of the objective function's H\"older's modulus.  

\begin{theorem}[Convex sets and functions; proof in companion paper \cite{2020oliveira:thompsonI}]\label{thm:convexrandomset} Make Assumption \ref{assump:LSCQ} with constants $\eta_*,\stheta_*$. Also assume $\parametersgoodnessplus$-goodness over the set $Z=X^{*,\stheta}_{\delta}$ for every choice of $(\stheta,\delta)\in [0,\stheta_*]\times [0,\eta_*]$ (cf. Assumption \ref{assump:goodFi}), where $\sigma^2(\stheta,\delta)$ depends continuously on $\stheta$ and $\delta$ (note that $\sigma^2(\stheta,\delta)$ depends on $(\stheta,\delta)$ but the other parameters in Assumption \ref{assump:goodFi} are fixed). 

Set $\hat\delta=0$. Fix parameter $t\geq 0$. For every $0<\epsilon\leq \stheta_*$ and $0<\delta<\eta_*$ satisfying $\epsilon+{\rm gap}(-\delta)\leq \stheta_*$, set:
\begin{eqnarray*}\widehat{w}_N(t;\delta;\epsilon)&:=&  \,\sigma(\epsilon+{\rm gap}(-\delta),\delta)\left\{4\sqrt{3}\frac{\gamma^{(\alpha)}_2(X^{*,\epsilon+{\rm gap}(-\delta)})}{\sqrt{N}}\right. \\ & & \left. + 6\sqrt{3}\frac{\diam^{\alpha}(X^{*,\epsilon+{\rm gap}(-\delta)})\sqrt{1 + \log(2|\mathcal{I}|+2)+t}}{\sqrt{N}}\right\},\end{eqnarray*} 
For $(\epsilon,\delta)$ as above, we define parameters $\check{\delta}(t;\epsilon)$ and $\check{w}(t;\epsilon)$ as follows. 
\begin{enumerate}\item If $\mathcal{I}=\emptyset$ (there are no constraints in expectation), then $\check{\delta}(t;\epsilon):=0$ and $\check{w}(t;\epsilon) = \widehat{w}_N(t;0;\epsilon)$.
\item Otherwise, assume that \[S_{N,\eta_*}(t;\epsilon):=\left\{\delta\in(0,\eta_*) \,:\, \begin{array}{l}\epsilon + {\rm gap}(-\delta) \leq \stheta_*,\\ \widehat{w}_N(t;\delta;\epsilon)+ \sigma_*\sqrt{\frac{6(1+\log(2|\mathcal{I}|+2)  + t)}{N}}< \delta\end{array}\right\}\]
is nonempty, and define \[\check{\delta}(t;\epsilon):=\inf S_{N,\eta_*}(t;\epsilon)\mbox{ and }\check{w}(t;\epsilon) := \widehat{w}_N(t;\check{\delta}(t;\epsilon);\epsilon).\]\end{enumerate}

Now, fix $\epsilon_0\in[0,\vartheta_*)$ and assume the set 
\[R_{N,\eta_*}(t;\epsilon_0):=\{\epsilon \in (\epsilon_0,\stheta_*]\,:\, \epsilon > \epsilon_0 + {\rm gap}(-\check{\delta}(t;\epsilon)) +  2\check{w}(t;\epsilon)\},\]
is nonempty so that 
$$
\check{r}(t;\epsilon_0):=\inf R_{N,\eta_*}(t;\epsilon_0)
$$ 
is well defined. Also set 
$$\check{\delta}(t):=\lim_{\epsilon\searrow \check{r}(t;\epsilon_0)}\check{\delta}(t;\epsilon).
$$

Now define ${\rm Good}_{\rm Thm. \ref{thm:convexrandomset}}(t,\epsilon_0)$ as the event where the following properties all hold.
\begin{enumerate}
\item[{\bf (a)}] \[\widehat{X}^{*,\epsilon_0}\subset X_{\check{\delta}(t)}^{*,\check{r}(t;\epsilon_0)};\]
that is, all $x\in \widehat{X}$ with $\widehat{F}(x)\leq \widehat{F}+\epsilon_0$ also satisfy $f(x)\leq f^*+\check{r}(t;\epsilon_0)$ and $\max_{i\in\mathcal{I}}f_i(x)\leq \check{\delta}(t)$;
\item[{\bf (b)}] the values of the SAA and the ideal problem satisfy:
\[|\widehat{F}^* - f^*|\leq \sigma_*\sqrt{\frac{6(1+\log(2|\mathcal{I}|+2)  + t)}{N}} + \frac{\check{r}(t;\epsilon_0)}{2}
+\max\{\check{r}(t;\epsilon_0),\gap(\check\delta(t))\}.
\]

\item[{\bf (c)}] for all $x\in \widehat{X}^{*,\epsilon_0}$, \[\dist(x,X)\leq \min\left\{\frac{\diam(X^{*,\stheta_*})\,\check{\delta}(t)}{\eta_*},2\diam(X^{*,\check{r}(t;\epsilon_0)})\right\}.\]
\end{enumerate}
Then \[\prob({\rm Good}_{\rm Thm. \ref{thm:convexrandomset}}(t,\epsilon_0))\geq 1-e^{-t}-2(|\mathcal{I}|+1)\rho.\] If we assume instead $\parametersgreatnessplus$-greatness of the functions $F_i$  (cf. Assumption \ref{assump:greatFi}) instead of $\parametersgoodnessplus$-goodness, then one may take $\rho = (\cbdg\kappa_p\sqrt{p/N})^p$ above.
\end{theorem}

\begin{remark}
For further comments, see Section 5.1 and remarks after Theorem 3 in our companion paper \cite{2020oliveira:thompsonI}.The non-emptyness of the sets $S_{N,\eta_*}(\cdot;\cdot)$ and $R_{N,\eta_*}(\cdot;\cdot)$, required in Theorem \ref{thm:convexrandomset}, are equivalent to a lower bound on the sample size $N$. As with the obtained rates, a localized lower bound on $N$ can be obtained by using the control on the quantities  $\gap(-\delta)$, $\sigma(\epsilon+\gap(-\delta),\delta)$, $\gamma_2^{(\alpha)}(X^{*,\epsilon+\gap(-\delta)})$ and $\diam(X^{*,\epsilon+\gap(-\delta)})$ and solving the inequalities defining $S_{N,\eta_*}(\cdot;\cdot)$ and $R_{N,\eta_*}(\cdot;\cdot)$.  This is more concretely exemplified in Propositions 3-4 in our companion paper \cite{2020oliveira:thompsonI} and also in Section \ref{sec:cvar} when applying Theorem \ref{thm:convexrandomset} for the risk-averse optimization problem. 
\end{remark}

\section{Risk-averse portfolio optimization}\label{sec:cvar}

In this section we illustrate our results to portfolio risk minimization with conditional value-at-risk constraints. One fundamental tool is the localized bound of Theorem \ref{thm:convexrandomset}. Let us recall that the conditional value at risk at probability level $p$ of an integrable random variable $Z$ is the quantity:
\[\cvar_p(Z) := \inf_{t\in \re}\,\left\{t + \frac{\esp{(Z-t)_+}}{p}\right\}.\]
Equivalently, $\cvar_p(Z)$ is the expectation of $Z$ when it is conditioned to lie above its $(1-p)$-th quantile. The conditional value at risk is one of the so-called coherent risk measures, which are applied in Finance due to their desirable properties \cite{rockafellar:urysaev2000}. 

Consider a random vector $\xi\in \re^{d_0}$ whose law $\probn$ has finite first moment. The coordinates 
of $\xi$ describe the losses of $d_0$ distinct assets $\xi[j]$ ($1\leq i\leq d_0$) over a period of time; negative losses correspond to gains. One must decide what fraction $x[i]$ of resources to allocate to each asset so as to maximize the expected return. The vector $x=(x[1],\dots,x[d_0])^T$ lives in the standard simplex in $\re^{d_0}$:
\[Y_0:= \left\{x\in \re^{d_0}_+\,:\, \sum_{j=1}^{d_0}x[j]=1\right\},\] 
and our goal will be to minimize the expected loss $\probn\langle x,\cdot\rangle$ over $x\in Y_0$ subject to a constraint over the conditional value at risk:
\begin{eqnarray}
\min_{x\in Y_0}&\quad & \probn\left[\langle x,\cdot\rangle\right]\nonumber\\
\mbox{s.t.}&\quad &\cvar_p(\langle x,\xi\rangle) \leq \beta.\label{problem:cvar:portfolio1}\end{eqnarray}

To make this fit the framework of \S \ref{sub:basic}, we set $d:=d_0+1$ and write elements of $\re^d$ as $(x,t)$ with $x\in\re^{d_0}$ and $t\in\re$. We set $Y:=Y_0\times \re$, $\mathcal{I}=\{1\}$ and define functions:
\[F(x,t,\xi):= \langle x,\xi\rangle\mbox{ and }F_1(x,t,\xi):=  t + \frac{{(\langle x,\xi\rangle-t)_+}}{p} - \beta.\]
Letting $f(x,t):= \probn\,F(x,t,\cdot)$ and $f_1(x,t):=\probn\,F_1(x,t,\cdot)$, we see that (\ref{problem:ideal}) is an equivalent formulation of the above problem. Here, 
\begin{align*}
X:=\left\{(x,t)\in Y:t+\frac{\esp[(\langle x,\xi\rangle-t)_+]}{p}-\beta\le 0\right\}. 
\end{align*}

We will assume that we have i.i.d. samples from $\xi\sim \probn$. This is probably not a realistic assumption in Finance, but is a good way to illustrate our general theory. In principle, the most general portfolio optimization problem would associate ``high risk'' to every asset with high returns. Still, it makes sense to envision a case in which, among these, there is a portion of them with significant more risk than the other remaining portion. We formalize this in the next assumption.

\begin{assumption}\label{assump:saacvar}The coordinates of random vector $\xi\sim \probn$ take the form:
\[\xi[j] := \left\{\begin{array}{lcl}-\mu+v[j] & , & 1\leq j\leq g;\\ 
v[j] & , & g+1\leq j\leq g+m;\\ -\mu + \phi W + v[j] & , & g+m+1\leq j\leq d_0.\end{array}\right.\]
Here:
\begin{enumerate}
\item $g,m\geq 1$ are positive integers with $g+m<d_0$;
\item $\mu>0$ is a positive constant; 
\item $v = (v[1],\dots,v[d_0])^T\in \re^{d_0}$ is a vector of independent and identically distributed random variables with mean $0$, variance $1$. We also assume that $\esp[\|v\|^2_\infty]\leq a$ and $\esp[\|v\|^4_\infty]\leq \kappa^2\,a^2$ for some $a,\kappa>0$. 
\item finally, $W\in \re$ is a random variable independent of vector $v$ with mean $0$, $\phi\ge1$,
$\esp[W^{2}]\leq w<+\infty$, $\esp[W^4]\leq \kappa^2 w^2$, and ${\cvar}_p[W]\geq \beta_*$, where 
$\kappa>0$ is defined in item 3.
\end{enumerate}
\end{assumption}
The random vector $\xi$ represents a caricature of a scenario where certain assets have highly correlated risks. The first $g$ assets in $\xi$ are high-return, low-risk assets. Their average return is $-\mu$ and the losses/gains (given by the $v[j]$) have relatively light tails. The next $m$ assets are also low-risk, but have significantly smaller expected gains. The last $d_0-m-g$ assets yield as large expected returns as the first $g$ ones, but they also carry a correlated ``risk factor" $\phi W$ with large conditional value-at-risk $\phi\beta_*$. Here, $\beta_*$ is the ``normalized risk'' and $\phi$ is the ``risk inflation'' factor when $W$ is multiplied by $\phi$. We highlight that, in Assumption \ref{assump:saacvar}, the set of ``good'' $g$ assets are associated to the first $g$ coordinates of $\xi$ for simplicity. Of course, the user does not know which these coordinates are.  

Even when Assumption \ref{assump:saacvar} holds, it is not obvious if the SAA of \eqref{problem:cvar:portfolio1} will behave well statistically (when $\phi$ or $\beta_*$ are large compared to other parameters). Precisely, we could in principle expect that the large parameters $(\phi,\beta_*)$ could affect the rate proportionally to the ``extrinsic dimension'' $d=d_0+1$. The content of the next Theorem \ref{thm:cvar} is to show that this is not the case: $(\phi,\beta_*)$ only affects the statistical rate via a term that is proportional to $g$, the number of ``good'' assets. We next explain what we mean by this exactly before the precise rates of Theorem \ref{thm:cvar} are presented.

Let $x^*$ be the solution of \eqref{problem:cvar:portfolio1} and $\widehat{x}^*$ be the solution of its SAA. To understand this theorem, note first that for a large class of distributions for $v$ and $W$, we may take $\kappa$ relatively small, say
$\kappa=10$. If other problem parameters are kept fixed, Theorem \ref{thm:cvar} implies that there exist constants $c,C>0$ depending on $p,\mu,\beta,\beta_*$ and $w$ such that, if  $N\geq d\,(\phi^2+a)$ then
\begin{eqnarray*}\probn\langle \widehat{x}^*,\cdot\rangle &\leq & \probn\langle x^*,\cdot\rangle + c\,\frac{(\phi^2 + a)^{1/2} \sqrt{1+t+g} +(\phi^2+a)^{1/2}\frac{\beta+\mu}{\phi\beta_*}\sqrt{d-g}}{\sqrt{N}},\\
\cvar_p(\langle \widehat{x}^*,\xi\rangle) &\leq & \beta +  c\,\frac{(\phi^2 + a)^{1/2} \sqrt{1+t+g} +(\phi^2+a)^{1/2}\frac{\beta+\mu}{\phi\beta_*}\sqrt{d-g}}{\sqrt{N}},
\end{eqnarray*}
with probability  $1-e^{-t}-C/N$. The above inequalities establish the rate of convergence of the optimal value and feasibility constraint of the SAA solution $\widehat{x}^*$ (compared to the true solution $x^*$). We have modeled a correlated ``risk factor'' $\phi W$ with a variance proportional to $\phi^2$ and conditional value-at-risk of at least $\phi\beta_*$. Examining the above rates, we see they can be split into the components of order
$$
\sqrt{\frac{(\phi^2+a)g}{N}}
\quad
\mbox{and}
\quad
\frac{\beta+\mu}{\phi\beta_*}\sqrt{\frac{(\phi^2+a)(d-g)}{N}}.
$$

We illustrate the above rates in two ``high risk'' regimes. First, suppose the risk inflation factor is reasonable, say, $\phi=1$. In this case, the dependence of the rate with $d-g$ decreases proportionally to $\beta_*^2$. In particular, if the normalized risk $\beta_*$ is very large, say, $\beta_*\ge(\beta+\mu)\sqrt{\frac{d-g}{g}}$, then the optimality and feasibility rates of the SAA are of the order  
$
\sqrt{\frac{(\phi^2+a)g}{N}}.
$
This rate depends only on $g$. Secondly, suppose now that the normalized risk $\beta_*$ is reasonable, say $\beta_*=1$, but the risk inflation factor $\phi$ is very large (compared to the other parameters). Then the optimality and feasibility rates of the SAA are of the order  
$
\sqrt{\frac{(\phi^2+a)g}{N}}
+(\beta+\mu)\sqrt{\frac{d-g}{N}}.
$
While this rate still depends on the extrinsic dimension $d$ it satisfies a ``robust'' property with respect to risk inflation: $\phi^2$ affects the rate via a term proportional to $g$ but \emph{not} to $d-g$. In establishing such precise rates, we crucially use the ``localization'' feature of Theorem \ref{thm:convexrandomset}.

We now state the theorem. 
\begin{theorem}\label{thm:cvar} Let $x^*$ denote the solution \eqref{problem:cvar:portfolio1} and $\widehat{x}^*$ denote the solution of its SAA problem. Let $A:=a+\phi^2 w+1$. Then there exist absolute constants $C,C_0>0$ such that, for all $t\geq 0$, if $-\mu + (1/p\sqrt{g})\leq \beta/2$ and
\begin{align*}
C_0\sqrt{A}\frac{\sqrt{g} + \left(\frac{1}{\mu} +  \frac{\beta+\mu}{\phi\beta_*}\right)\sqrt{d-g}+(\Delta+2)\sqrt{1+t}}{p\sqrt{N}} <\beta/2,\\
C_0\frac{\sqrt{A}\sqrt{d-g}}{p\sqrt{N}\mu}\le1/2,
\end{align*}
where $\Delta:=\frac{\beta+\mu}{1-p}$, then the following holds with probability $\geq 1-e^{-t}-C\kappa^2/N$: 
\begin{eqnarray*}\probn\langle \widehat{x}^*,\cdot\rangle &\leq & \probn\langle x^*,\cdot\rangle + C\frac{\sqrt{A}}{p\sqrt{N}}\left\{\sqrt{g} + \frac{\beta+\mu}{\phi\beta_*}\sqrt{d-g} + (\Delta+2) \sqrt{1+t}\right\},\\
\cvar_p(\langle  \widehat{x}^*,\xi\rangle) &\leq & \beta +C\frac{\sqrt{A}}{p\sqrt{N}}\left\{\sqrt{g} + \frac{\beta+\mu}{\phi\beta_*}\sqrt{d-g} + (\Delta+2) \sqrt{1+t}\right\}\\
& & + C\frac{\sqrt{A}\sqrt{d-g}}{p\sqrt{N}\mu}\cdot\frac{\sqrt{A}}{p\sqrt{N}}\left\{\sqrt{g} + \frac{\beta+\mu}{\phi\beta_*}\sqrt{d-g} + (\Delta+2) \sqrt{1+t}\right\}.
\end{eqnarray*}
\end{theorem}
\begin{remark}
We note that, from Proposition \ref{prop:solution}, $\probn\langle x^*,\cdot\rangle=-\mu$.
\end{remark}
 
\subsection{Preliminaries}

We will prove Theorem \ref{thm:cvar} via Theorem \ref{thm:convexrandomset}. To do this, we first prove some properties about the corresponding ideal problem. One fact that we will often use implicitly is that the conditional value-at-risk of a random variable $Z$ satisfies:
\[\forall a\in \re_+,\, \forall b\in\re \,:\, \cvar_p(aZ+b) = a \cvar_p(Z) + b.\]

In the rest of this section, $\|\cdot\|$ denotes the $\ell_1$ norm over $\re^d$. 
\begin{proposition}\label{prop:valuemax}For any $x\in Y_0$,
\[\probn[\langle x,\cdot\rangle] = -\mu\,\left(1-\sum_{j=g+1}^{g+m}x[j]\right).\]
In particular, the value of the problem (\ref{problem:cvar:portfolio1}) is at least $-\mu$ and any solution $x$ with 
$\esp[\langle \xi,x\rangle]\leq \epsilon-\mu$ must satisfy
\[\sum_{j=g+1}^{g+m}x[j]\leq \frac{\epsilon}{\mu}.\]\end{proposition}
\begin{proof}Note that
\begin{equation}\label{eq:decompinner}\langle x,\xi\rangle  = \langle x,v\rangle - \mu\,\left(1 - \sum_{i=g+1}^{g+m}x[i]\right) + W \left(\sum_{i=g+m+1}^{d_0}x[i]\right).\end{equation}
Since $\esp[ v]=0$ and $\esp[W]=0$, we obtain the expression for $\esp[\langle \xi,x\rangle]$. The other statements follow trivially.
\qed
\end{proof}

\begin{proposition}\label{prop:solution} Assume  $-\mu + 1/(p\sqrt{g})\leq \beta/2$. Choose \[(x^*,t^*)\in Y\times \re\] with $x^*[j]=1/g$ for $1\leq j\leq g$, $x^*[j]=0$ for $j>g$, and $t^*$ a $(1-p)$-quantile of $-\langle x^*,\xi\rangle$. Then $(x^*,t^*)$ is an optimal solution to (\ref{problem:cvar:portfolio1}), and 
$\cvar_p[\langle \xi,x^*\rangle]\leq \beta /2$.\end{proposition} 
\begin{proof}Clearly, $\probn \langle x^*,\cdot\rangle=-\mu$. Moreover, 
\[\cvar_p(\langle x^*,\xi\rangle) = \cvar_p\left(-\mu + \frac{1}{g}\sum_{j=1}^{g}v[j]\right)\leq -\mu +\frac{\esp\left|\frac{1}{g}\sum_{j=1}^gv[j]\right|}{p}.\]
Since 
\[\esp\left|\frac{1}{g}\sum_{j=1}^gv[j]\right|\leq \sqrt{\esp\left(\frac{1}{g}\sum_{j=1}^gv[j]\right)^2} = \frac{1}{\sqrt{g}},\]
we deduce (using our assumptions) that 
$\cvar_p(\langle x^*,\xi\rangle) \leq -\mu + 1/(p\sqrt{g})\leq \beta/2$ and $(x^*,t^*)$ is feasible.
\qed
\end{proof}

\begin{proposition}\label{prop:CVaRnotbig} For any pair $(x,t)\in Y$,
\[\sum_{j=g+m+1}^{d_0}x[j]\leq \frac{\cvar_p(\langle x,\xi\rangle)+ \mu}{\phi\beta_*}.\]\end{proposition}
\begin{proof}Looking at (\ref{eq:decompinner}) and the fact that $v$ and $W$ are independent and centered, we see that:
\[\esp[\langle x,\xi\rangle\mid W] = -\mu\,\left(1 - \sum_{j=g+1}^{g+m}x[j]\right) + \phi W \left(\sum_{j=g+m+1}^{d_0}x[j]\right).\]
In particular, for any $t\in \re$, the conditional Jensen's inequality implies:
\begin{equation*}\esp(\langle x,\xi\rangle - t)_+ \geq \esp\left(-\mu\,\left(1 - \sum_{j=g+1}^{g+m}x[j]\right) + \phi W\left(\sum_{j=g+m+1}^{d_0}x[j]\right)- t\right)_+,\end{equation*}
and we obtain:
\begin{eqnarray*}\cvar_p(\langle x,\xi\rangle)&\geq & -\left(1 - \sum_{j=g+1}^{g+m}x[j]\right)\,\mu + \left(\sum_{j=g+m+1}^{d_0}x[j]\right)\,\cvar_p(\phi W)\\ &\geq & -\mu + \phi\beta_*\,\left(\sum_{j=g+m+1}^{d_0}x[j]\right).
\end{eqnarray*}
\qed
\end{proof}

\begin{proposition}\label{prop:t:feasible}
Given a feasible pair $(x,t)\in X$, we have
\[\beta - \Delta \leq t\leq \beta,\mbox{ with }\Delta:=\frac{\beta+\mu}{1-p}.\]
 \end{proposition}
 \begin{proof} Write 
\begin{equation}
\tilde F_1(x,t,\xi):= t + \frac{(\langle x,\xi\rangle - t)_+}{p}.
\end{equation} 
 If $(x,t)$ is feasible, $\probn[\tilde F_1(x,t,\cdot)]\leq \beta$, so $t\leq \beta$. Moreover, 
\begin{eqnarray*}\probn[\tilde F_1(x,t,\cdot)]&\geq & t + \probn\left[\frac{(\langle x,\cdot\rangle - t)}{p}\right] \\ &=& \probn\left[\frac{\langle x,\cdot\rangle }{p}\right] - \frac{(1-p)\,t}{p}\\ &\geq & -\frac{\mu}{p} - \frac{(1-p)\,t}{p},\end{eqnarray*}
so
\[-\frac{\mu}{p} - \frac{(1-p)\,t}{p}\leq \beta\Rightarrow t\geq \beta -\Delta.\]\qed
\end{proof}

\begin{lemma}\label{lem:covernumber} Assume $\|\cdot\|$ is the $\ell^1$ norm over $\re^d$. Then:
\begin{itemize}
\item[i)] The 
$\parametersgreatnessq$-greatness assumption is satisfied over $Y$ (cf. Assumption \ref{assump:greatFi}) with 
\[\sigma_*^2 = \sigma^2 = \frac{2A}{p^2} = \frac{a + \phi^2 w + 1}{p^2},\,\alpha=1,\, q=2\mbox{ and }\kappa_2=\kappa.\]
As a consequence, the $(\sigma_* ^2,\sigma^2,1,\rho)$-goodness assumption over $Y$ is also satisfied with
\[\rho:= \frac{2}{N}\left(\cbdg\kappa\right)^2.\]
\item[ii)] Given $\epsilon>0$, 
$\diam(X^{*,\epsilon})\leq \Delta+2$ and
\[\gamma_2^{(1)}(X^{*,\epsilon})\lesssim \sqrt{g} +\min\left\{1,\left(\frac{\epsilon}{\mu} + \frac{\beta+\mu}{\phi\beta_*}\right)\right\}\,\sqrt{d-g}  + \Delta.\]\end{itemize}
\end{lemma}
\begin{proof}We fix the $\ell_1$ norm over $\re^d=\re^{d_0+1}$. With this notation, it is clear that $x\mapsto F_1(x,t,\xi)$ is a 
$\mathsf{L}_1(\xi)$-Lipschitz function of $x$, with
\[\mathsf{L}_1(\xi):=\frac{\phi|W| + \|v\|_{\infty} + 1}{p}.\]

Under our assumptions,
\[\probn\mathsf{L}^2_1(\cdot) \leq \frac{2(\phi^2w + a + 1)}{p^2} = \frac{2A}{p^2}\mbox{ and }  \probn\mathsf{L}^4_1(\cdot)\leq 4\kappa^2\,\frac{\phi^4w^2 + a^2 + 1}{p^4} \le \kappa^2\,\left(\frac{2A}{p^2}\right)^2,\]
so the $\parametersgreatnessq$-greatness assumption is indeed satisfied as claimed in item i).

Let us now prove item ii). Given $m\in\mathbb{N}$, let $\mathbb{B}^m$ be the $\ell_1$ unit ball in $\re^m$. Propositions  \ref{prop:valuemax}, \ref{prop:CVaRnotbig} and \ref{prop:t:feasible} give:
\[X^{*,\epsilon}\subset \mathbb{B}^g\times \left\{\min\left[1,\left(\frac{\epsilon}{\mu} + \frac{\beta+\mu}{\phi\beta_*}\right)\right]\,\mathbb{B}^{d-g}\right\}\times [\beta-\Delta,\beta].\]
This fact, Proposition \ref{prop:sizeproduct} and (\ref{eq:simpleboundgamma2}) give the claimed bound on $\gamma_2^{(1)}(X^{*,\epsilon})$ (with room to spare in the constant factor).
Finally, for any $(x,t),(x',t')\in X^{*,\epsilon}$, $
\Vert x-x'\Vert=\sum_{j=1}^{d_0}|x[j]-x'[j]|+|t-t'|
\le2+\Delta,
$
since $x$ and $x'$ lie in the unit simplex and $|t-t'|\le\Delta$.
We thus conclude that $\diam(X^{*,\epsilon})\le2+\Delta$.
\qed\end{proof}

\subsection{Proof of SAA approximation properties} 

We are now ready to prove our main result on the $\cvar$-constrained portfolio optimization problem.  

\begin{proof}[of Theorem \ref{thm:cvar}] In this proof, $C,C_0>0$ denote absolute constants whose value may change from line to line. We will apply Theorem \ref{thm:convexrandomset} in what follows, and we take the notation in that theorem for granted. Let us fix $t\ge0$. In our setting 
$\epsilon_0:=0$. The assumptions of Theorem \ref{thm:convexrandomset} are satisfied with $|\mathcal{I}|=1$, 
$\stheta_*:=1$ and 
$\eta_*=\beta/2$ (say). Indeed, the optimal solution $(x^*,t^*)$ from Proposition \ref{prop:solution} belongs to $X_{-\eta_*}^{\stheta_*}$ under the assumptions of the present theorem and the feasible set is compact. Therefore,  ${\rm gap}(-\delta)=0$ for all $0\leq \delta\leq \eta_*$. This ensures Assumption \ref{assump:LSCQ} is satisfied. As for Assumption \ref{assump:greatFi}, by Lemma \ref{lem:covernumber} we may take
$\sigma_*^2=\sigma^2(\stheta,\delta)=2A/p^2$ and 
$\rho$ as in item i) of Lemma \ref{lem:covernumber}. It follows that the $(\sigma_* ^2,\sigma^2(\stheta,\delta),1,\rho)$-goodness assumption is satisfied over $Y$ and so over $X_{\delta}^{\stheta}$ for all $\stheta\in[0,\stheta^*]$ and $\delta\in[0,\eta_*]$.

We assume $N$ satisfies
\begin{align}
C_0\sqrt{A}\frac{\sqrt{g} + \left(\frac{1}{\mu} +  \frac{\beta+\mu}{\phi\beta_*}\right)\sqrt{d-g}+(\Delta+2)\sqrt{1+t}}{p\sqrt{N}} <\beta/2,\label{eq:assumeN:eq1}\\
C_0\frac{\sqrt{A}\sqrt{d-g}}{p\sqrt{N}\mu}\le1/2,\label{eq:assumeN:eq2}.
\end{align}
By item ii) in Lemma \ref{lem:covernumber}, the quantity 
$\widehat{w}_N(t;\delta;\epsilon)$ in Theorem \ref{thm:convexrandomset} satisfies:
\begin{align}
\widehat{w}_N(t;\delta;\epsilon)\leq C\sqrt{A}\,\frac{\sqrt{g} + \left(\frac{\epsilon}{\mu} + \frac{\beta+\mu}{\phi\beta_*}\right)\,\sqrt{d-g} + (\Delta+2) \sqrt{1+t}}{p\sqrt{N}},\label{eq:cvar:wN:upper:bound}
\end{align}
for all $0\leq \delta\leq \eta_*$ and $0\leq \epsilon\leq \stheta^*$. Recalling that $\sigma_*^2=2A/p^2$ and by adjusting the constant $C_0$ in (\ref{eq:assumeN:eq1}) if necessary, we may ensure that, for all $0\leq \epsilon\leq \stheta^*$,  $S_{N,\eta_*}(t;\epsilon)\neq \emptyset$ and 
\begin{align}
\check{\delta}(t;\epsilon)\leq C\sqrt{A}\,\frac{\sqrt{g} + \left(\frac{\epsilon}{\mu} + \frac{\beta+\mu}{\phi\beta_*}\right)\,\sqrt{d-g} + (\Delta+2) \sqrt{1+t}}{p\sqrt{N}}.\label{eq:cvar:check:delta:upper:bound}
\end{align}

Consider now the definition of $R_{N,\eta_*}(t;0)$ and recall that $\gap(-\check\delta(t,\epsilon))=0$ for all $0\le\epsilon\le\stheta_*=1$. As, of course, $\check w_N(t;\epsilon):=\widehat{w}_N(t;\check\delta(t;\epsilon);\epsilon)$ satisfies the same upper  bound \eqref{eq:cvar:wN:upper:bound}, we may adjust $C_0$ in (\ref{eq:assumeN:eq1}) if necessary and obtain that $R_{N,\eta_*}(t;0)\neq \emptyset$. By definition, in establishing an upper bound on 
$\check r(t;0)$, it is sufficient to find the least $\epsilon\in(0,\stheta_*]$ satisfying
\begin{align*}
2C\frac{\sqrt{A}}{p\sqrt{N}}\left\{\sqrt{g} + \frac{\beta+\mu}{\phi\beta_*}\sqrt{d-g} + (\Delta+2) \sqrt{1+t}\right\}<\left(1-2C\frac{\sqrt{A}}{p\sqrt{N}}\cdot\frac{\sqrt{d-g}}{\mu}\right)\epsilon.
\end{align*}

Therefore, by adjusting the constant in \eqref{eq:assumeN:eq2} if necessary,
\[\check{r}(t;0)\leq  
C\frac{\sqrt{A}}{p\sqrt{N}}\left\{\sqrt{g} + \frac{\beta+\mu}{\phi\beta_*}\sqrt{d-g} + (\Delta+2) \sqrt{1+t}\right\}.\]

Finally, by setting 
$\epsilon\searrow\check{r}(t;0)$ in \eqref{eq:cvar:check:delta:upper:bound},  we get
\begin{eqnarray*}\check{\delta}(t)&\leq &  C\frac{\sqrt{A}}{p\sqrt{N}}\left\{\sqrt{g} + \frac{\beta+\mu}{\phi\beta_*}\sqrt{d-g} + (\Delta+2) \sqrt{1+t}\right\}\\
& & + C\frac{\sqrt{A}\sqrt{d-g}}{p\sqrt{N}\mu}\cdot\frac{\sqrt{A}}{p\sqrt{N}}\left\{\sqrt{g} + \frac{\beta+\mu}{\phi\beta_*}\sqrt{d-g} + (\Delta+2) \sqrt{1+t}\right\}.\qed
\end{eqnarray*} 

\end{proof}

\section{Persistency results for least squares with LASSO-type constraints}\label{subsection:lasso}
In this section we obtain improved persistence bounds for a variant of LASSO method for least squares in very high dimensions.  One main building block is Proposition \ref{lem:convexlocalized}. The sample space considered is $\Xi:=\re^d\times\re$ and a point $\xi\in\Xi$ will be decomposed as $\xi=(\mathbf{x}(\xi),y(\xi))$ where $\mathbf{x}(\xi)\in\re^d$ and $y(\xi)\in\re$. Within the setup of Section \ref{sec:basic}, we define the loss function
$$
F(x,\xi):=\left[y(\xi)-\langle\mathbf{x}(\xi),x\rangle\right]^2\quad\quad(x\in\re^d),
$$
where $\langle\cdot,\cdot\rangle$ denotes the standard inner product. In our notation, we define the \emph{risk} $f(x):=\probn F(x,\cdot)$ and  \emph{empirical risk}
$\widehat F(x):=\widehat\probn F(x,\cdot)$, where $\widehat\probn$ is the empirical distribution with respect to a size-$N$ i.i.d. sample $\xi^N$ of $\probn$. Finally, we let $Y\subset\re^d$ denote a closed convex set and consider minimizing $f$ over (a subset of) $Y$ via estimators $\widehat x:=\widehat x(\xi^N)$. We define the \emph{population} and \emph{empirical design matrices}, $\mathbf{\Sigma}\in\re^{d\times d}$ and $\mathbf{\widehat\Sigma}\in\re^{d\times d}$, respectively, by
\begin{eqnarray}
\forall v\in\re^d,\quad\mathbf{\Sigma} v&:=&\probn\langle v,\mathbf{x}(\cdot)\rangle\mathbf{x}(\cdot),\label{def:population:matrix}\\
\mathbf{\widehat\Sigma} v&:=&\widehat\probn\langle v,\mathbf{x}(\cdot)\rangle\mathbf{x}(\cdot).\label{def:design:matrix}
\end{eqnarray}

Recall that the usual ordinary least squares method minimizes $\widehat F$. When $N\gg d$,
this method typically produces a good approximation of the minimizer of $f$. This is
not true in the $N\ll d$ setting, where the least squares estimator is not consistent.
For this setting, Tibshirani \cite{tibshirani1996} proposed minimizing $\widehat F$ subject to a constraint on the $\ell_1$ norm $\Vert x\Vert_1$ of $x$: for some $R>0$,
$$
\widehat x^R_{\mbox{\tiny{lasso}},0}:=\argmin\left\{\widehat F(x):x\in\re^d,\Vert x\Vert_1\le R\right\}.
$$

Since then there has been an explosion of theoretical and practical work on the
LASSO. Most of the current literature considers a penalized variant of this estimator. Let's denote the $\ell$-th coordinate of $\mathbf{x}(\xi)$ as $\mathbf{x}(\xi)[\ell]$ and the diagonal matrix in $\re^{d\times d}$ with entries $a_1,\ldots,a_p$ as $\diag(a_\ell)_{\ell=1}^d$. Given $q\in[1,\infty)$, we define the following diagonal matrices in $\re^{d\times d}$:
\begin{eqnarray}
\mathbf{\widehat D}_{q}:=\diag\left(\sqrt[q]{\widehat\probn\left|\mathbf{x}(\cdot)[\ell]\right|^q}\right)_{\ell=1}^d,\quad\quad
\mathbf{D}_{q}:=\diag\left(\sqrt[q]{\probn\left|\mathbf{x}(\cdot)[\ell]\right|^q}\right)_{\ell=1}^d.
\label{def:pen:diag:matrix:empirical}
\end{eqnarray}
It is instructive to remark that the diagonal elements of the matrices in \eqref{def:population:matrix}-\eqref{def:pen:diag:matrix:empirical} are related by $\mathbf{\Sigma}[\ell,\ell]=\mathbf{D}_2[\ell,\ell]^2$ and $\mathbf{\widehat \Sigma}[\ell,\ell]=\mathbf{\widehat D}_2[\ell,\ell]^2$. 
Bickel, Ritov and Tsybakov \cite{2009bickel:ritov:tsybakov} considers the penalized least squares problem under fixed design:
$$
\widehat x^R_{\mbox{\tiny{lasso}},1}:=\argmin\left\{\widehat F(x)+\lambda\Vert\mathbf{\widehat D}_2 x\Vert_1:x\in\re^d\right\}.
$$

\begin{theorem}[A persistent result for LASSO-type constraints with heavier tails]\label{thm:lasso}
Assume $(\mathbf{x}(\xi),y(\xi))\in\re^d\times\re$ is a random vector with finite $q$th moments, $q\ge9$. Considering definition \eqref{def:population:matrix}, we assume that there exist numbers $C,u>0$ and $p\in(0,1]$ such that
\begin{eqnarray}
\forall v\in\re^d,&&\quad\probn\left\{\xi\in\Xi:|\langle v,\mathbf{x}(\xi)\rangle|>u\sqrt{\langle v,\mathbf{\Sigma} v\rangle}\right\}\ge p,\label{equation:thm:lasso:small:ball:prob}\\
\forall 1\le\ell\le d,&&\quad\sqrt[q]{\probn|\mathbf{x}(\cdot)[\ell]|^q}\le C\sqrt[3]{\probn|\mathbf{x}(\cdot)[\ell]|^3}.
\label{equation:thm:lasso:moment:condition}
\end{eqnarray}

Let $\widehat\probn$ be the empirical distribution corresponding to a size-$N$ i.i.d. sample $\{\xi_j\}_{j\in[N]}$ of $\probn$. Choose $R>0$ and define: 
$$
\widehat x_{\mbox{\tiny{\emph{lasso}}}}:=\argmin_{x\in\re^d}\left\{\widehat F(x):\Vert\mathbf{\widehat D}_3 x\Vert_1\le R\right\}.
$$
Choose also the confidence level $\delta\in(0,1)$. 

Then there exists $C_0>0$ depending only on $C$, $q$, $u$ and $p$ such that the following holds. Suppose that the sample size is large enough so that $N\ge(1/\delta)^{\frac{1}{(q/6)-1}}$ and
$$
\alpha:=C_0\sqrt{\frac{\ln(d/\delta)}{N}}\le\frac{1}{2}.
$$
We define the ``true solution''
$$
x_*:=\argmin_{x\in\re^d}\left\{f(x):\Vert\mathbf{D}_3 x\Vert_1\le(1+\alpha)R\right\},
$$
and the ``noise'' $\epsilon(\xi):=y(\xi)-\langle x_*,\mathbf{x}(\xi)\rangle$. Then with probability $\ge1-\delta$, the following exact feasibility and near-optimality hold for the \emph{SAA} solution $\widehat x_{\emph{\tiny{lasso}}}$:
\begin{eqnarray}
\Vert\mathbf{D}_3\widehat x_{\mbox{\tiny{\emph{lasso}}}} \Vert_1&\le &(1+\alpha)R,\label{thm:lasso:feasibility}\\
f(\widehat x_{\mbox{\tiny{\emph{lasso}}}})-f(x_*)&\le & C_1\left\{\left[(\probn+\widehat\probn)\epsilon(\cdot)^6\right]^{\frac{1}{6}}R\sqrt{\frac{\ln(d/\delta)}{N}}+R^2\frac{\ln(d/\delta)}{N}\right\},\label{thm:lasso:optimality}
\end{eqnarray}
where $C_1>0$ is a constant that only depends on $C$ and $q$.
\end{theorem}

\subsection{Proof of Theorem \ref{thm:lasso}}
The proof of Theorem \ref{thm:lasso} will be derived as a consequence of the following four probabilistic Lemmas \ref{lemma:norm}-\ref{lemma:quad} and the verification that these lemmas imply, with high-probability, the conditions of Proposition \ref{lem:convexlocalized}. The proofs of such lemmas are postponed to Subsection \ref{subsection:proof:prob:lemmas}.

\begin{lemma}\label{lemma:norm}
Let $\alpha$ as in Theorem \ref{thm:lasso} with $C_0$ only depending on $C$. Then the event
$$
\mathsf{Norm}:=\{\Vert\mathbf{\widehat D}_{3}x_*\Vert_1\le(1+\alpha)^2R\},
$$
has probability $\prob(\mathsf{Norm})\ge1-\frac{\delta}{4}$.
\end{lemma}

\begin{lemma}\label{lemma:diag}
Let $\alpha\in(0,\frac{1}{2}]$ be defined as in Theorem \ref{thm:lasso}. Then, by enlarging $C_0$ (as a function only of $C$), the event
$$
\mathsf{Diag}:=\bigcap_{\ell=1}^d\left\{\widehat\probn|\mathbf{x}(\cdot)[\ell]|^3\ge\frac{1}{(1+\alpha)^3}\probn|\mathbf{x}(\cdot)[\ell]|^3\right\},
$$
has probability $\prob(\mathsf{Diag})\ge1-\frac{\delta}{4}$.
\end{lemma}

\begin{lemma}\label{lemma:grad}
By enlarging $C_0$ (as a function only of $C$ as stated in Theorem \ref{thm:lasso}), there exists a constant $C_2>0$, depending only on $C$, such that the event
$$
\mathsf{Grad}:=\left\{\left\Vert\mathbf{\widehat D}_3^{-1}(\probn-\widehat\probn)\epsilon(\cdot)\mathbf{x}(\cdot)\right\Vert_\infty\le C_2\sqrt{\frac{\ln(d/\delta)}{N}}\sqrt[6]{(\probn+\widehat\probn)\epsilon(\cdot)^6}\right\},
$$
has probability $\prob(\mathsf{Grad})\ge1-\frac{\delta}{4}$.
\end{lemma}
\begin{lemma}\label{lemma:quad}
By enlarging $C_0$ (as a function of $C$, $u$ and $p$ as stated in Theorem \ref{thm:lasso}), there exists constants $C_3>0$ and $\phi\in(0,1)$, depending only on $C$, $u$ and $p$, such that the event
$$
\mathsf{Quad}:=\left\{\forall v\in\re^d, \langle v,\mathbf{\widehat{\Sigma}}v\rangle\ge\phi\langle v,\mathbf{\Sigma}v\rangle-C_3\frac{\ln(d/\delta)}{N}\Vert\mathbf{\widehat D}_3v\Vert_1\right\},
$$
has probability $\prob(\mathsf{Quad})\ge1-\frac{\delta}{4}$.
\end{lemma}
	
Granting the above lemmas, our proof strategy is to prove the following claim:

\begin{quote}	
\textbf{Claim}: Whenever the events $\mathsf{Norm}$, $\mathsf{Diag}$, $\mathsf{Quad}$ and $\mathsf{Grad}$ all take place, the inequalities \eqref{thm:lasso:feasibility}-\eqref{thm:lasso:optimality} are valid.
\end{quote}
	
\emph{From this point on we will assume that $\mathsf{Norm}\cap\mathsf{Diag}\cap\mathsf{Quad}\cap\mathsf{Grad}$ takes place} with the constants $C_0$, $C_1$ and chosen $R$, $\alpha$, $N$ and $\delta$ as stated in Theorem \ref{thm:lasso} (leaving their proof to Subsection \ref{subsection:proof:prob:lemmas}). Our goal is then to prove the \textbf{Claim} since it implies that
$$
\prob\left\{\mbox{\eqref{thm:lasso:feasibility} and \eqref{thm:lasso:optimality} hold}\right\}\ge\prob\left\{\mathsf{Norm}\cap\mathsf{Diag}\cap\mathsf{Quad}\cap\mathsf{Grad}\right\}\ge1-\delta,
$$
which in turn implies Theorem \ref{thm:lasso}. We will prove \textbf{Claim} by verifying the conditions of the Proposition \ref{lem:convexlocalized}, which is an deterministic perturbation result. 

In our case we take $Y:=\re^d$ and let 
\begin{eqnarray*}
f(x)&:=&\probn(y(\cdot)-\langle\mathbf{x}(\cdot),x\rangle)^2,\\
f_1(x)&:=&\Vert\mathbf{D}_3x\Vert_1-(1+\alpha)R,\quad\quad(x\in\re^d)
\end{eqnarray*}
be, respectively, the objective and the unique soft constraint. We thus have $X:=\{x\in\re^d:\Vert\mathbf{D}_3x\Vert_1\le(1+\alpha)R\}$. Our noisy objective and constraint are, respectively,
\begin{eqnarray*}
\widehat F(x)&:=&\widehat\probn(y(\cdot)-\langle\mathbf{x}(\cdot),x\rangle)^2,\\
\widehat F_1(x)&:=&\Vert\mathbf{\widehat D}_3x\Vert_1-(1+\alpha)R,\quad\quad(x\in\re^d),
\end{eqnarray*}
with the tuning parameter
$
\hat\delta:=-\alpha R,
$
so that
$$
\widehat X:=\{x\in\re^d:\Vert\mathbf{\widehat D}_3x\Vert_1\le R\}.
$$

To continue, we state the following formulae obtained from the first-order condition of problem $\min_X f$.
\begin{lemma}\label{lemma:KKT}
The following identities hold:
\begin{eqnarray*}
f(x)-f(x_*)&=&\langle(x-x_*),\mathbf{\Sigma}(x-x_*)\rangle-2\langle\probn\epsilon(\cdot)\mathbf{x}(\cdot),x-x_*\rangle,\\
\widehat F(x)-\widehat F(x_*)&=&\langle(x-x_*),\mathbf{\widehat\Sigma}(x-x_*)\rangle-2\langle\widehat\probn\epsilon(\cdot)\mathbf{x}(\cdot),x-x_*\rangle.
\end{eqnarray*}
Moreover, for all $x\in X$,
\begin{eqnarray*}
\langle\probn\epsilon(\cdot)\mathbf{x}(\cdot),x-x_*\rangle\le0,\quad\quad
\langle(x-x_*),\mathbf{\Sigma}(x-x_*)\rangle\le f(x)-f(x_*).
\end{eqnarray*}
\end{lemma}
\begin{proof}	
To prove the first identity, just note that, by definition of $\epsilon(\xi)$, we have that $f(x)-f(x_*)=\probn(\epsilon(\cdot)+\langle\mathbf{x}(\cdot),x_*-x\rangle)^2-\probn\epsilon(\cdot)^2$. To finish, just expand such relation and use that $\langle v,\mathbf{\Sigma}v\rangle=\probn\langle\mathbf{x(\cdot)},v\rangle^2$ for any $v\in\re^d$. The proof is analogous for the second identity replacing the measure $\probn$ by $\widehat\probn$. We omit the details.

From the first identity in the statement, we conclude that the first-order condition of problem $\min_Xf$ is:
$$
\langle2\mathbf{\Sigma}(x-x_*)-2\probn\epsilon(\cdot)\mathbf{x}(\cdot),x-x_*\rangle\ge0,\quad\forall x\in X.
$$
The above relation and, again, the first identity in the statement prove the inequalities stated in the lemma.\qed
\end{proof}

It will be useful to bound the quantities in the previous lemma in terms of norms of $\mathbf{D}_3x$ and $\mathbf{\widehat D}_3x$ for given $x\in\re^d$. This is the content of the next lemma.
\begin{lemma}\label{lemma:bounds:in:terms:D3x}
For all $x\in\re^d$,
\begin{eqnarray*}
\langle x,\mathbf{\Sigma}x\rangle &\le &\Vert\mathbf{D}_3x\Vert_1^2,\\
\langle x,\mathbf{\widehat\Sigma}x\rangle &\le &\Vert\mathbf{\widehat D}_3x\Vert_1^2,\\
\langle\probn\epsilon(\cdot)\mathbf{x}(\cdot),x\rangle &\le &\sqrt[6]{\probn\epsilon(\cdot)^6}\cdot\Vert\mathbf{D}_3x\Vert_1.
\end{eqnarray*}
\end{lemma}
\begin{proof}
In order to prove the first inequality, we let $z:=\mathbf{D}_3x$ and note that
\begin{eqnarray*}
\langle x,\mathbf{\mathbf{\Sigma}}x\rangle =\langle z,\mathbf{D}_3^{-1}\mathbf{\Sigma}\mathbf{D}_3^{-1}z\rangle
=\sum_{\ell,k=1}^dz[\ell]z[k]\left\{\frac{\probn(\mathbf{x}(\cdot)[\ell]\mathbf{x}(\cdot)[k])}{\sqrt[3]{\probn|\mathbf{x}(\cdot)[\ell]|^3}\cdot\sqrt[3]{\probn|\mathbf{x}(\cdot)[k]|^3}}\right\},
\end{eqnarray*}
where we have used that $\mathbf{D}_3^{-1}$ is diagonal with $i$-th entry $(\probn|\mathbf{x}(\cdot)[i]|^3)^{-\frac{1}{3}}$ and $\mathbf{\Sigma}[\ell,k]=\probn(\mathbf{x}(\cdot)[\ell]\mathbf{x}(\cdot)[k])$. We claim that the terms in curly brackets above are less than 1. To prove this, we just need to apply the Cauchy-Schwarz and Jensen inequalities:
$$
\probn(\mathbf{x}(\cdot)[\ell]\mathbf{x}(\cdot)[k])\le[\probn|\mathbf{x}(\cdot)[\ell]|^2]^{\frac{1}{2}}
[\probn|\mathbf{x}(\cdot)[k]|^2]^{\frac{1}{2}}
\le [\probn|\mathbf{x}(\cdot)[\ell]|^3]^{\frac{1}{3}}
[\probn|\mathbf{x}(\cdot)[k]|^3]^{\frac{1}{3}}.
$$
We deduce that
$$
\langle x,\mathbf{\Sigma}x\rangle\le\sum_{\ell,k=1}^d z[\ell]z[k]=\Vert z\Vert_1^2=\Vert\mathbf{D}_3x\Vert_1,
$$
as claimed. The proof of the second inequality stated in the lemma is analogous. 

For the last inequality stated in the lemma, we use a similar reasoning. We have
\begin{eqnarray*}
\langle\probn\epsilon(\cdot)\mathbf{x}(\cdot),x\rangle &=&
\langle\mathbf{D}_3^{-1}\probn\epsilon(\cdot)\mathbf{x}(\cdot),z\rangle\\
&\le &\sum_{\ell=1}^d|z[\ell]|\frac{|\probn(\epsilon(\cdot)\mathbf{x}(\cdot)[\ell])|}{[\probn|\mathbf{x}(\cdot)[\ell]|^3]^{\frac{1}{3}}}\\
(\mbox{by H\"older's inequality})\quad &\le &\sum_{\ell=1}^d|z[\ell]|[\probn|\epsilon(\cdot)[\ell]|^{\frac{2}{3}}]^{\frac{3}{2}}\\
(\mbox{by Jensen's inequality})\quad &\le &\Vert z\Vert_1[\probn|\epsilon(\cdot)|^6]^{\frac{1}{6}}=
\Vert \mathbf{D}_3x\Vert_1[\probn|\epsilon(\cdot)|^6]^{\frac{1}{6}}.
\end{eqnarray*}
We have thus finished the proof.\qed
\end{proof}

We now prepare the ground for applying Proposition \ref{lem:convexlocalized} for appropriate $\delta^\circ$, 
$\delta$, $\epsilon$. In our setting $\epsilon_0=0$ as we are interested in the exact Lasso solution. We set $\delta:=0$. We first find $y_*\in X$ such that
$
f_1(y_*)<-\alpha R.
$
We claim that it suffices to choose
\begin{equation}
y_*:=\frac{1-\alpha}{1+\alpha}x_*.\label{equation:y*}
\end{equation}
Indeed $y_*\in X$ and $\Vert\mathbf{D}_3y_*\Vert_1-(1+\alpha)R=\frac{1-\alpha}{1+\alpha}\Vert\mathbf{D}_3x_*\Vert_1-(1+\alpha)R\le(1-\alpha)R-(1+\alpha)R=-2\alpha R$. We now set 
$\delta^\circ:=-f_1(y_*)>0$ and $x_{-\delta^\circ}:=y_*$.

If \textsf{Norm} holds and $\alpha\le\frac{1}{2}$ (under the conditions of Theorem \ref{thm:lasso}), then we also have the following useful inequality
\begin{eqnarray}
\Vert\mathbf{D}_3(y_*-x_*)\Vert_1\bigvee\Vert\mathbf{\widehat D}_3(y_*-x_*)\Vert_1\le3\alpha R.
\label{equation:D3(y*-x*)}
\end{eqnarray}
Indeed, 
$
\Vert\mathbf{D}_3(y_*-x_*)\Vert_1=\frac{2\alpha}{1+\alpha}\Vert\mathbf{D}_3 x_*\Vert_1
\le 2\alpha R,
$ 
since $x_*\in X$ and 
$
\Vert\mathbf{\widehat D}_3(y_*-x_*)\Vert_1=\frac{2\alpha}{1+\alpha}\Vert\mathbf{\widehat D}_3 x_*\Vert_1 
\le 2\alpha(1+\alpha)R\le3\alpha R,
$
since \textsf{Norm} holds and $\alpha\le\frac{1}{2}$. 

From Lemmas \ref{lemma:KKT}-\ref{lemma:bounds:in:terms:D3x} and \eqref{equation:D3(y*-x*)}, we get
\begin{eqnarray}
f(y_*)-f(x_*)&\le &\Vert\mathbf{D}_3(y_*-x_*)\Vert_1^2+2\Vert\mathbf{D}_3(y_*-x_*)\Vert_1\sqrt[6]{\probn\epsilon(\cdot)^6}\nonumber\\
&\lesssim &\alpha^2R^2+\alpha R \Vert_1\sqrt[6]{\probn\epsilon(\cdot)^6}\nonumber\\
&\le &C_4\left(\frac{\ln(d/\delta)}{N}R^2+\sqrt{\frac{\ln(d/\delta)}{N}}R\sqrt[6]{\probn\epsilon(\cdot)^6}\right)\le\epsilon,\label{equation:f(y*)-f(x*)}
\end{eqnarray}
for some constant $C_4$ depending on $C_0$. In above, the last inequality is chosen by definition of $\epsilon$ so that $\epsilon\ge f(y_*)-f^*$ as required by Proposition \ref{lem:convexlocalized}.

We will now check conditions \eqref{eq:controlconstraintsinterior}-\eqref{eq:controlobjective} of Proposition \ref{lem:convexlocalized}. We start with \eqref{eq:controlconstraintsinterior}: since $\delta^\circ=-f_1(y_*)$ it is enough to obtain 
$\widehat F_1(y_*)\le-\alpha R=\hat\delta$. Indeed,
\begin{eqnarray*}
\widehat F_1(y_*)&=&\Vert\mathbf{\widehat D}_3y_*\Vert_1-(1+\alpha)R\\
&=&\frac{1-\alpha}{1+\alpha}\Vert\mathbf{\widehat D}_3x_*\Vert_1-(1+\alpha)R\\
(\mbox{$\mathsf{Norm}$ holds})\quad\quad &\le &(1-\alpha)(1+\alpha)R-(1+\alpha)R\\
&=&-\alpha(1+\alpha)R<-\alpha R.
\end{eqnarray*}

We now check \eqref{eq:controlconstraintsboundary}. As $\delta=0$ and $\hat\delta=-\alpha R$, it is enough to prove the stronger bound
$$
\sup_{x\in\re^d}\{-\widehat F_1(x):f_1(x)=0\}\le\alpha R.
$$ 
With our definition of $\widehat F_1$ and $f_1$, this is tantamount to proving that for all $x\in\re^d$ with $\Vert\mathbf{D}_3x\Vert_1=(1+\alpha)R$, we have $\Vert\mathbf{\widehat D}_3x\Vert_1\ge R$. Indeed, this follows from the fact that on the event $\mathsf{Diag}$,
$$
\forall x\in\re^d,\quad\Vert\mathbf{\widehat D}_3x\Vert_1\ge\frac{\Vert\mathbf{D}_3x\Vert_1}{(1+\alpha)}.
$$

To finish the proof, we must also check \eqref{eq:controlobjective}. Recall $\delta=0$. It is actually sufficient to give a stronger upper bound on the quantity
\begin{align}
I:=-\hat\Delta_0(x_*,y_*)+\sup\{-\hat\Delta_0(x,x_*):x\in X\cap\widehat X,f(x)-f^*=\epsilon\}.\label{eq:sup:final:condition}
\end{align}
We start by bounding $-\hat\Delta_0(x_*,y_*)$. We have	
\begin{eqnarray}
-\hat\Delta_0(x_*,y_*)&=&\widehat F(y_*)-\widehat F(x_*)-(f(y_*)-f(x_*))\nonumber\\
(\mbox{by Lemma \ref{lemma:KKT}})\quad &=&\langle(y_*-x_*),(\mathbf{\widehat\Sigma}-\mathbf{\Sigma})(y_*-x_*)\rangle-2\langle(\widehat\probn-\probn)\epsilon(\cdot)\mathbf{x}(\cdot),y_*-x_*\rangle\nonumber\\
(\mbox{by Lemma \ref{lemma:bounds:in:terms:D3x}})\quad &\le &\Vert\mathbf{\widehat D}_3(y_*-x_*)\Vert_1^2+\Vert\mathbf{D}_3(y_*-x_*)\Vert_1^2\nonumber\\
&&+2\Vert\mathbf{\widehat D}_3^{-1}(\widehat\probn-\probn)\epsilon(\cdot)\mathbf{x}(\cdot)\Vert_\infty\Vert\mathbf{\widehat D}_3(y_*-x_*)\Vert_1\nonumber\\
(\mbox{by \eqref{equation:D3(y*-x*)}})\quad &\lesssim &\alpha^2R^2+\alpha R\Vert\mathbf{\widehat D}_3^{-1}(\widehat\probn-\probn)\epsilon(\cdot)\mathbf{x}(\cdot)\Vert_\infty\nonumber\\
(\mbox{$\mathsf{Grad}$ holds})\quad &\lesssim &\alpha^2R^2+C_2\alpha R\sqrt{\frac{\ln(d/\delta)}{N}}\sqrt[6]{(\probn+\widehat\probn)\epsilon(\cdot)^6}\nonumber\\
&\le &C_5\left\{\frac{\ln(d/\delta)}{N}R^2+\frac{\ln(d/\delta)}{N}R\sqrt[6]{(\probn+\widehat\probn)\epsilon(\cdot)^6}\right\},\label{equation:hat:delta(y*,x_*)}
\end{eqnarray}
for some constant $C_5$ depending on $C_0$.

Now we bound the sup in \eqref{eq:sup:final:condition}. 
For all $x\in X\cap\widehat X$ with $f(x)-f(x_*)=\epsilon$,
\begin{eqnarray}
-\hat\Delta_0(x,x_*)&= & f(x)-f(x_*)-(\widehat F(x)-\widehat F(x_*))\nonumber\\
(\mbox{by Lemma \ref{lemma:KKT}})\quad &=&\langle(x-x_*),(\mathbf{\Sigma}-\mathbf{\widehat\Sigma})(x-x_*)\rangle-2\langle(\probn-\widehat\probn)\epsilon(\cdot)\mathbf{x}(\cdot),x-x_*\rangle\nonumber\\
&\le &\sup\left\{\langle h,(\mathbf{\Sigma}-\mathbf{\widehat\Sigma})h\rangle:\langle h,\mathbf{\Sigma}h\rangle\le \epsilon,\quad\Vert\mathbf{\widehat D}_3 h\Vert_1\le 5R\right\}\nonumber\\
&+&2\sup\left\{\langle h,(\probn-\widehat\probn)\epsilon(\cdot)\mathbf{x}(\cdot)\rangle:\Vert\mathbf{\widehat D}_3h\Vert_1\le 5R\right\}=:I_1+I_2,\nonumber\\
\label{equation:hat:Delta(x,x*)}
\end{eqnarray}
where in last inequality we used $\langle x-x_*,\mathbf{\Sigma}(x-x_*)\rangle\le f(x)-f(x_*)$ by Lemma \ref{lemma:KKT} and the fact that $\Vert\mathbf{\widehat D}_3(x-x_*)\Vert_1\le \Vert\mathbf{\widehat D}_3x\Vert_1+\Vert\mathbf{\widehat D}_3x_*\Vert_1\le R+(1+\alpha)^2R\le 5R$ since $x\in\widehat X$, the event $\mathsf{Norm}$ holds and $0\le\alpha\le1$. 

The first term $I_1$ in \eqref{equation:hat:Delta(x,x*)} may be bounded as
\begin{eqnarray}
(\mbox{by $\mathsf{Quad}$})\quad I_1 &\le & \sup\left\{(1-\phi)\langle h,\mathbf{\Sigma}h\rangle+C_3\frac{\ln(d/\delta)}{N}\Vert\mathbf{\widehat D}_3h\Vert_1:\langle h,\mathbf{\Sigma}h\rangle\le \epsilon,\quad\Vert\mathbf{\widehat D}_3 h\Vert_1\le 5R\right\}\nonumber\\
&\le &(1-\phi)\epsilon+5C_3R\frac{\ln(d/\delta)}{N}.\label{equation:hat:Delta(x,x*):I1}
\end{eqnarray}
The second term $I_2$ in \eqref{equation:hat:Delta(x,x*)} may be bounded as 
\begin{eqnarray}
I_2&\le &2\sup\left\{\Vert\mathbf{\widehat D}_3h\Vert_1\cdot\Vert\mathbf{\widehat D}_3^{-1}(\probn-\widehat\probn)\epsilon(\cdot)\mathbf{x}(\cdot)\Vert_\infty:\Vert\mathbf{\widehat D}_3 h\Vert_1\le 5R\right\}\nonumber\\
(\mbox{by $\mathsf{Grad}$})\quad &\le &10C_2R\sqrt{\frac{\ln(d/\delta)}{N}}\sqrt[6]{(\probn+\widehat\probn)\epsilon(\cdot)^6}.
\label{equation:hat:Delta(x,x*):I2}
\end{eqnarray}
Hence, from \eqref{equation:hat:Delta(x,x*)}-\eqref{equation:hat:Delta(x,x*):I2} we finally get that the sup in \eqref{eq:sup:final:condition} is upper bounded by
\begin{eqnarray}
(1-\phi)\epsilon+5C_3R\frac{\ln(d/\delta)}{N}+10C_2R\sqrt{\frac{\ln(d/\delta)}{N}}\sqrt[6]{(\probn+\widehat\probn)\epsilon(\cdot)^6}.\label{equation:hat:Delta(x,x*):final}
\end{eqnarray}

To finalize the proof of \eqref{eq:controlobjective}, we just need to show
\begin{align}
I+f(y_*)-f^*<\epsilon.\label{eq:controlobjective:aux}
\end{align}
Let $I_3:=(R\vee R^2)\frac{\ln(d/\delta)}{N}+R\sqrt{\frac{\ln(d/\delta)}{N}}\sqrt[6]{(\probn+\widehat\probn)\epsilon(\cdot)^6}$. We observe from \eqref{equation:f(y*)-f(x*)} (which gives an lower bound for $\epsilon$), \eqref{equation:hat:delta(y*,x_*)} and \eqref{equation:hat:Delta(x,x*):final}, that we may choose 
$\epsilon=\mathcal{O}(I_3)$ in order to \eqref{eq:controlobjective:aux} to hold. Condition \eqref{eq:controlobjective} is checked. 

From Proposition \ref{lem:convexlocalized}, we conclude that 
$\widehat{X}^*\subset X^{*,\epsilon}$. This means that $\widehat x_{\emph{\tiny{lasso}}}$ satisfies conditions \eqref{thm:lasso:feasibility}-\eqref{thm:lasso:optimality} of the \textbf{Claim} under the conditions of $\alpha$, $N$, $R$ and $\delta$ of Theorem \ref{thm:lasso}. The \textbf{Claim} is proved.

\subsubsection{Proofs of probabilistic lemmas}\label{subsection:proof:prob:lemmas}

As mentioned in the previous section, we need to prove Lemmas \ref{lemma:norm}-\ref{lemma:quad} in order to complete the proof of Theorem \ref{thm:lasso}. This is carried out in this section.
\begin{proof}[of Lemma \ref{lemma:norm}]
Without loss on generality, we may assume $\mathbf{D}_3x_*\neq0$. Since $\Vert\mathbf{D}_3x_*\Vert_1\le(1+\alpha)R$, it is sufficient to prove
\begin{equation}\label{equation:lemma:norm:eq1}
\prob\left\{\frac{\Vert\mathbf{\widehat D}_3x_*\Vert_1}{\Vert\mathbf{D}_3x_*\Vert_1}\ge1+\alpha\right\}\le\prob\left\{\left(\frac{\Vert\mathbf{\widehat D}_3x_*\Vert_1}{\Vert\mathbf{D}_3x_*\Vert_1}\right)^3\ge1+\alpha\right\}\le\frac{\delta}{4},
\end{equation}
the first inequality above being trivial. 

Set $z:=\frac{\mathbf{D}_3x_*}{\Vert\mathbf{D}_3x_*\Vert_1}$ so that $\Vert z\Vert_1=1$. By the definition of $z$, convexity and \eqref{def:pen:diag:matrix:empirical}, 
\begin{eqnarray}
\left(\frac{\Vert\mathbf{\widehat D}_3x_*\Vert_1}{\Vert\mathbf{D}_3x_*\Vert_1}\right)^3-1&=&
\left(\sum_{\ell=1}^d\frac{\sqrt[3]{\widehat\probn|\mathbf{x}(\cdot)[\ell]|^3}\cdot |x_*[\ell]|}{\Vert\mathbf{D}_3x_*\Vert_1}\right)^3-1\nonumber\\
&=&\left(\sum_{\ell=1}^d\frac{|z[\ell]|\sqrt[3]{\widehat\probn|\mathbf{x}(\cdot)[\ell]|^3}}{\sqrt[3]{\probn|\mathbf{x}(\cdot)[\ell]|^3}}\right)^3-1\nonumber\\
&\le &\left(\sum_{\ell=1}^d\frac{|z[\ell]|\widehat\probn|\mathbf{x}(\cdot)[\ell]|^3}{\probn|\mathbf{x}(\cdot)[\ell]|^3}\right)-1\nonumber\\
&=&\sum_{\ell=1}^d|z[\ell]|\frac{\widehat\probn|\mathbf{x}(\cdot)[\ell]|^3-\probn|\mathbf{x}(\cdot)[\ell]|^3}{\probn|\mathbf{x}(\cdot)[\ell]|^3}=\widehat\probn g(\cdot),\label{equation:lemma:norm:eq2}
\end{eqnarray}
where we have defined the function
$$
\xi\mapsto g(\xi):=\sum_{\ell=1}^d|z[\ell]|\frac{|\mathbf{x}(\xi)[\ell]|^3-\probn|\mathbf{x}(\cdot)[\ell]|^3}{\probn|\mathbf{x}(\cdot)[\ell]|^3}.
$$
Note that $\esp[g(\xi)]=0$. 

We next estimate the variances $\probn g(\cdot)^2$ and $\widehat\probn g(\cdot)^2$. We have
\begin{eqnarray}
\probn g(\cdot)^2&=&\esp\left[\left(\sum_{\ell=1}^d|z[\ell]|\frac{|\mathbf{x}(\xi)[\ell]|^3-\probn|\mathbf{x}(\cdot)[\ell]|^3}{\probn|\mathbf{x}(\cdot)[\ell]|^3}\right)^2\right]\nonumber\\
(\mbox{by convexity})\quad\quad &\le &\sum_{\ell=1}^d|z[\ell]|\esp\left[\left(\frac{|\mathbf{x}(\xi)[\ell]|^3-\probn|\mathbf{x}(\cdot)[\ell]|^3}{\probn|\mathbf{x}(\cdot)[\ell]|^3}\right)^2\right]\nonumber\\
&\le &\Vert z\Vert_1\max_{\ell\in[d]}\frac{\probn|\mathbf{x}(\cdot)[\ell]|^6}{(\probn|\mathbf{x}(\cdot)[\ell]|^3)^2}\le c_0,\label{equation:lemma:norm:eq3}
\end{eqnarray}
for some constant $c_0$ depending only on $C$, where we have used that $\sqrt[6]{\probn|\mathbf{x}(\cdot)[\ell]|^6}\le\sqrt[q]{\probn|\mathbf{x}(\cdot)[\ell]|^q}$ and \eqref{equation:thm:lasso:moment:condition}. 

Following the same guideline above, we get 
\begin{eqnarray}
\widehat\probn g(\cdot)^2&\le &\max_{\ell\in[d]}\frac{\widehat\probn(|\mathbf{x}(\cdot)[\ell]|^3-\probn|\mathbf{x}(\cdot)[\ell]|^3)^2}{(\probn|\mathbf{x}(\cdot)[\ell]|^3)^2}\nonumber\\
&=&\max_{\ell\in[d]}\frac{\widehat\probn|\mathbf{x}(\cdot)[\ell]|^6+(\probn|\mathbf{x}(\cdot)[\ell]|^3)^2-2\widehat\probn|\mathbf{x}(\cdot)[\ell]|^3\cdot\probn|\mathbf{x}(\cdot)[\ell]|^3}{(\probn|\mathbf{x}(\cdot)[\ell]|^3)^2}\nonumber\\
&\le &\max_{\ell\in[d]}\frac{\widehat\probn|\mathbf{x}(\cdot)[\ell]|^6}{(\probn|\mathbf{x}(\cdot)[\ell]|^3)^2}+1.\label{equation:lemma:norm:eq4}
\end{eqnarray}
From Markov's inequality for the random variable $Z:=|\mathbf{x}(\cdot)[\ell]|^6$ with $q$ replaced by $\mathfrak{p}:=q/6$, we obtain that
\begin{eqnarray}
\prob\left\{\widehat\probn Z\le 
\probn Z+c_{\mathfrak{p}}\frac{\Lpcnorm{Z-\probn Z}}{N^{1-\frac{1}{\mathfrak{p}}}}(8/\delta)^{\frac{1}{\mathfrak{p}}}
\right\}\ge1-\frac{\delta}{8},\label{equation:lemma:norm:eq5}
\end{eqnarray}
where $c_{\mathfrak{p}}$ is a constant that only depends on $\mathfrak{p}$. From \eqref{equation:lemma:norm:eq4}-\eqref{equation:lemma:norm:eq5} and $N\ge(\delta^{-1})^{1/(\mathfrak{p}-1)}$, we obtain that, with probability $\ge1-\frac{\delta}{8}$, 
\begin{eqnarray}
\widehat\probn g(\cdot)^2&\le &
\max_{\ell\in[d]}\left\{\frac{(1+8^{1/\mathfrak{p}}c_{\mathfrak{p}})\probn|\mathbf{x}(\cdot)[\ell]|^6}{(\probn|\mathbf{x}(\cdot)[\ell]|^3)^2}
+8^{1/\mathfrak{p}}c_{\mathfrak{p}}\frac{\left(\probn|\mathbf{x}(\cdot)[\ell]|^q\right)^{\frac{6}{q}}}{(\probn|\mathbf{x}(\cdot)[\ell]|^3)^2}\right\}+1\nonumber\\
&\le &C_{\mathfrak{p}},\label{equation:lemma:norm:eq6}
\end{eqnarray}
for some constant $C_{\mathfrak{p}}$ depending only on $C$ and $\mathfrak{p}=q/6$, where we have used that $\sqrt[6]{\probn|\mathbf{x}(\cdot)[\ell]|^6}\le\sqrt[q]{\probn|\mathbf{x}(\cdot)[\ell]|^q}$ and \eqref{equation:thm:lasso:moment:condition}.

First, from \eqref{equation:lemma:norm:eq3} and \eqref{equation:lemma:norm:eq6}, we obtain that, with probability $\ge1-\delta/8$, the bound $\hat\sigma^2:=(\widehat\probn+\probn)g(\cdot)^2\le C_{\mathfrak{p}}+c_0$ holds.  Secondly, the upper tail in Lemma \ref{lemma:concentration:ineq:self:norm} in the Appendix with $t:=\ln(16/\delta)<\ln(\mathcal{O}(1)d/\delta)$, implies that, with probability $\ge1-\delta/8$, $\widehat\probn g(\cdot)\le \sqrt{\frac{2(1+\ln(16/\delta))}{N}\hat\sigma^2}$. From the two previous facts, the union bound and proper definitions of $C_0$ and $\alpha$ stated in Theorem \ref{thm:lasso} in terms of $c_0$ and 
$C_{\mathfrak{p}}$ above, we get that
\begin{eqnarray}
\prob\left\{\widehat\probn g(\cdot)\ge\alpha\right\}\le\frac{\delta}{4}.\label{equation:lemma:norm:eq7}
\end{eqnarray}
Relations \eqref{equation:lemma:norm:eq2} and \eqref{equation:lemma:norm:eq7} prove the second inequality in \eqref{equation:lemma:norm:eq1} as desired.\qed
\end{proof}

\begin{proof}[of Lemma \ref{lemma:diag}]
For each $\ell\in[d]$, we will apply Lemma \ref{lemma:sub-gaussian:lower:tail} in the Appendix with $Z_1:=|\mathbf{x}(\xi_1)[\ell]|^3$ and $a:=2$. In our case, from \eqref{equation:thm:lasso:moment:condition} we have $\probn|\mathbf{x}(\cdot)[\ell]|^6\le\{\probn|\mathbf{x}(\cdot)[\ell]|^7\}^{\frac{6}{7}}\le 
C^6\{\probn|\mathbf{x}(\cdot)[\ell]|^3\}^{\frac{6}{3}}$, that is, $\frac{\esp[Z_1^2]}{(\esp[Z_1])^2}\le C^6$. From this fact, Lemma \ref{lemma:sub-gaussian:lower:tail} and the union bound over $\ell\in[d]$ we obtain
\begin{eqnarray*}
\prob\left\{\bigcap_{\ell\in[d]}\left(\widehat\probn|\mathbf{x}(\cdot)[\ell]|^3\ge(1-\epsilon)\probn|\mathbf{x}(\cdot)[\ell]|^3\right)\right\}\ge1-d \exp\left\{-\frac{\sqrt{\epsilon}}{2C^6}N\right\}.
\end{eqnarray*}

From the above relation, we observe that, in order to obtain the desired claim $\prob(\mathsf{Diag})\ge1-\frac{\delta}{4}$ with $\alpha$ defined in Theorem \ref{thm:lasso}, it is enough find a (deterministic) $\epsilon>0$ such that $\sqrt{\epsilon}\ge\frac{2C^6}{N}\ln(4d/\delta)$ and $1-\epsilon\ge\frac{1}{(1+\alpha)^3}\Leftrightarrow\epsilon\le\frac{\alpha^3+\alpha^2+\alpha}{(1+\alpha)^3}$. By increasing $C_0$ (as a function of $C$) in the definition of $\alpha$ in Theorem \ref{thm:lasso}, such property is satisfied if we can find a $\epsilon$ such that $\alpha^4<\epsilon\le\frac{\alpha^3+\alpha^2+\alpha}{(1+\alpha)^3}$. Since $0<\alpha\le\frac{1}{2}$ for large enough $N$, we obtain $\alpha^4(1+\alpha)^3<\alpha^3+\alpha^2+\alpha$ and, hence, such $\epsilon>0$ exists. The claim is proved.\qed 
\end{proof}

\begin{proof}[of Lemma \ref{lemma:grad}]
As a first step, for every $\ell\in[d]$, we will apply Lemma \ref{lemma:concentration:ineq:self:norm} in the Appendix to the random variable $g(\xi):=\epsilon(\xi)\mathbf{x}(\xi)[\ell]$ with $t:=\ln(\frac{32d}{\delta})$ (considering the upper and lower tails inequalities). From the obtained tail inequalities, the union bound and the fact that, for all $\ell\in[d]$,
\begin{eqnarray*}
(\widehat\probn+\probn)\left\{\epsilon(\cdot)\mathbf{x}(\cdot)[\ell]-\probn\epsilon(\cdot)\mathbf{x}(\cdot)[\ell]\right\}^2&\le& 2(\widehat\probn+\probn)\left\{\epsilon(\cdot)^2\mathbf{x}(\cdot)[\ell]^2+\probn\epsilon(\cdot)^2\mathbf{x}(\cdot)[\ell]^2\right\}\\
&=&2\widehat\probn\epsilon(\cdot)^2\mathbf{x}(\cdot)[\ell]^2+3\probn\epsilon(\cdot)^2\mathbf{x}(\cdot)[\ell]^2\\
&\le &3(\widehat\probn+\probn)\left\{\epsilon(\cdot)^2\mathbf{x}(\cdot)[\ell]^2\right\},
\end{eqnarray*}
we conclude that, with probability $\ge1-\frac{\delta}{8}$, the following holds: for all $\ell\in[d]$,
\begin{eqnarray}
|(\widehat\probn-\probn)\epsilon(\cdot)\mathbf{x}(\cdot)[\ell]|\le\sqrt{\frac{6(1+\ln(32d/\delta))}{N}(\widehat\probn+\probn)\left\{\epsilon(\cdot)^2\mathbf{x}(\cdot)[\ell]^2\right\}}.\label{lemma:grad:eq1}
\end{eqnarray}
We also note that H\"older's inequality guarantees that 
\begin{eqnarray}
\sqrt{(\widehat\probn+\probn)\left\{\epsilon(\cdot)^2\mathbf{x}(\cdot)[\ell]^2\right\}}
\le[(\widehat\probn+\probn)\epsilon(\cdot)^6]^{\frac{1}{6}}[(\widehat\probn+\probn)\mathbf{x}(\cdot)[\ell]^3]^{\frac{1}{3}}.\label{lemma:grad:eq2}
\end{eqnarray}

For each $\ell\in[d]$, we will now apply Lemma \ref{lemma:sub-gaussian:lower:tail} in the Appendix with $Z_1:=|\mathbf{x}(\xi_1)[\ell]|^3$, $a:=2$ and $\epsilon:=\frac{1}{2}$. As observed in the proof of Lemma \ref{lemma:diag}, we obtain $\frac{\esp[Z_1^2]}{(\esp[Z_1])^2}\le C^6$ from the condition \eqref{equation:thm:lasso:moment:condition}. From this fact, Lemma \ref{lemma:sub-gaussian:lower:tail} and the union bound over $\ell\in[d]$ we obtain that, with probability $\ge1-d\exp\{-\frac{N}{2\sqrt{2}C^6}\}$, the following holds: 
\begin{equation}\label{lemma:grad:eq3}
\forall\ell\in[d],\quad\quad\probn|\mathbf{x}(\cdot)[\ell]|^3\le2\widehat\probn|\mathbf{x}(\cdot)[\ell]|^3.
\end{equation}
Hence, by enlarging the constant $C_0$ as a function of $C$ in the definition of $\alpha$ in Theorem \ref{thm:lasso}, we obtain that \eqref{lemma:grad:eq3} holds with probability $\ge1-\frac{\delta}{8}$.

Combining \eqref{lemma:grad:eq1}-\eqref{lemma:grad:eq3} with an union bound, we deduce that, with probability $\ge\frac{\delta}{4}$, the following holds: 
\begin{eqnarray*}
\forall\ell\in[d],\quad\quad|(\widehat\probn-\probn)\epsilon(\cdot)\mathbf{x}(\cdot)[\ell]|\le C_2\sqrt{\frac{\ln(d/\delta)}{N}}
[(\widehat\probn+\probn)\epsilon(\cdot)^6]^{\frac{1}{6}}[\widehat\probn\mathbf{x}(\cdot)[\ell]^3]^{\frac{1}{3}},
\end{eqnarray*}
for some constant $C_2>0$ depending just on $C$. The above relation is exactly the property defining the event $\mathsf{Grad}$ since the $\ell$-th diagonal entry of $\mathbf{\widehat D}_3$ is $[\widehat\probn\mathbf{x}(\cdot)[\ell]^3]^{\frac{1}{3}}$. We have thus proved that $\prob(\mathsf{Grad})\ge1-\frac{\delta}{4}$.\qed
\end{proof}

\begin{proof}[of Lemma \ref{lemma:quad}]
Recall that $\{\xi_j\}_{j=1}^N$ is an i.i.d. sample from $\probn$ so that $\{\mathbf{x}(\xi_j)\}_{j\in[N]}$ are i.i.d. random vectors. Using definitions \eqref{def:population:matrix}-\eqref{def:design:matrix}, we have, for any $h\in\re^d$, 
\begin{eqnarray*}
\langle h,\mathbf{\Sigma}h\rangle=\probn\langle\mathbf{x}(\cdot),h\rangle^2,\quad\quad\quad
\langle h,\mathbf{\widehat\Sigma}h\rangle=\probn\langle\mathbf{x}(\cdot),h\rangle^2=
\frac{1}{N}\sum_{j=1}^N\langle\mathbf{x}(\xi_j),h\rangle^2.
\end{eqnarray*}
Related to the above expressions, we will now invoke the ``small-ball condition'' \eqref{equation:thm:lasso:small:ball:prob} and a related result of Lecu\'e and Mendelson stated in Corollary 2.5, item (2) of \cite{lecue:mendelson2017}. Such corollary implies that, under \eqref{equation:thm:lasso:small:ball:prob}, there exist universal constants $c_1, c_2>0$  such that, if $s\in\mathbb{N}\setminus\{0\}$, $N\ge c_1\frac{s\ln(ed/s)}{p^2}$ and $\phi:=1\wedge(u^2p/2)$, then\footnote{Actually, the Corollary 2.5 in \cite{lecue:mendelson2017} cover only the case when $\mathbf{\Sigma}$ is the identity matrix, but the arguments are based on VC dimension theory that are readily extendable to our setting. We omit such details.} 
\begin{eqnarray}
\prob\left\{\forall h\in\re^d, \quad\Vert h\Vert_0\le s\Rightarrow\langle h,\mathbf{\widehat\Sigma}h\rangle\ge\phi\langle h,\mathbf{\Sigma}h\rangle\right\}\ge1-e^{-c_2p^2N}.\label{lemma:quad:eq2}
\end{eqnarray}

Secondly, we will now invoke a result of Oliveira stated in Lemma 5.1 of \cite{oliveira2013}.\footnote{See also Corollary 2.5, item (2) of \cite{lecue:mendelson2017} for essentially the same statement. See also  \cite{oliveira2016}.} That lemma shows that the property
\begin{eqnarray}
\forall h\in\re^d, \quad\Vert h\Vert_0\le s\Rightarrow\langle h,\mathbf{\widehat\Sigma}h\rangle\ge\phi\langle h,\mathbf{\Sigma}h\rangle,\label{lemma:quad:eq3}
\end{eqnarray}
stated in the event in \eqref{lemma:quad:eq2} and \emph{restricted to $s$-sparse vectors}, actually implies the property
$$
\forall h\in\re^d, \quad\langle h,\mathbf{\widehat\Sigma}h\rangle\ge\phi\langle h,\mathbf{\Sigma}h\rangle-\frac{\Vert\mathbf{D}^{1/2}h\Vert_1}{s-1},
$$
\emph{extended to all vectors in $\re^d$} at a cost of an extra $\ell_1$ error in the lower bound, where $\mathbf{D}$ is the diagonal matrix with entries $\mathbf{D}[\ell,\ell]:=(\mathbf{\widehat\Sigma}[\ell,\ell]-\phi\mathbf{\Sigma}[\ell,\ell])_+$ and $\mathbf{D}^{1/2}$ is the square-root of $\mathbf{D}$. Since, 
\begin{eqnarray*}
\forall\ell\in\re^d,\quad\quad\mathbf{\widehat\Sigma}[\ell,\ell]-\phi\mathbf{\Sigma}[\ell,\ell]\le\mathbf{\widehat\Sigma}[\ell,\ell]=\mathbf{\widehat D}_2[\ell,\ell]^2\le \mathbf{\widehat D}_3[\ell,\ell]^2,
\end{eqnarray*}
we obtain from the two previous inequalities that 
\begin{equation}
\forall h\in\re^d, \quad\langle h,\mathbf{\widehat\Sigma}h\rangle\ge\phi\langle h,\mathbf{\Sigma}h\rangle-\frac{\Vert\mathbf{\widehat D}_3h\Vert_1}{s-1}.\label{lemma:quad:eq4}
\end{equation}

From \eqref{lemma:quad:eq2} and the fact that \eqref{lemma:quad:eq3}$\Rightarrow$\eqref{lemma:quad:eq4}, we obtain 
\begin{eqnarray*}
\prob\left\{\forall h\in\re^d,\quad\langle h,\mathbf{\widehat\Sigma}h\rangle\ge\phi\langle h,\mathbf{\Sigma}h\rangle-\frac{\Vert\mathbf{\widehat D}_3h\Vert_1}{s-1}\right\}\ge1-e^{-c_2p^2N}.
\end{eqnarray*}
Using the above relation with $s:=1+\frac{N}{C_3\ln(d/\delta)}$ for a sufficiently large constant $C_3>0$ and by enlarging $C_0$ in the statement of Theorem \ref{thm:lasso} as functions of $C$, $u$ and $p$, we obtain that $\prob(\mathsf{Quad})\ge1-\frac{\delta}{4}$ as desired.\qed
\end{proof}

\section*{Appendix}

\begin{proof}[of Proposition \ref{prop:sizeproduct}]Given admissible sequences $\{\mathcal{A}_{1,j}\}_{j\geq 0}$ and $\{\mathcal{A}_{2,j}\}_{j\geq 0}$ for $\mathcal{M}_1$ and $\mathcal{M}_2$, one may define an admissible sequence $\{\mathcal{C}_j\}_{j\geq 0}$ for $\mathcal{M}$ via:
\[\mathcal{C}_0:=\{\mathcal{M}\}\mbox{ and }\mathcal{C}_j:= \mathcal{A}_{1,j-1}\times \mathcal{A}_{2,j-1}\,(j\geq 1).\]
It is easy to see that this is indeed admissible and moreover
\[\diam(\mathcal{C}_0) = \diam(\mathcal{M})= \diam(\mathcal{M}_1) + \diam(\mathcal{M}_2),\]
\[\diam(\mathcal{C}_j)\leq \diam(\mathcal{A}_{1,j-1}) + \diam(\mathcal{A}_{2,j-1}).\]
Therefore,
\[\gamma_2^{(\alpha)}(\mathcal{M})\leq  \diam(\mathcal{M})^\alpha + \sum_{j\geq 1}2^{j/2}(\diam(\mathcal{A}_{1,j-1})^\alpha + \diam(\mathcal{A}_{2,j-1})^\alpha)\]
or equivalently
\begin{eqnarray*}\gamma_2^{(\alpha)}(\mathcal{M})&\leq &  \diam(\mathcal{M}_1)^\alpha + \diam(\mathcal{M}_2)^\alpha \\ & & + \sqrt{2}\left(\sum_{j\geq 0}2^{j/2}\diam(\mathcal{A}_{1,j-1})^\alpha\right) \\ & & +\sqrt{2}\left(\sum_{j\geq 0}2^{j/2}\diam(\mathcal{A}_{2,j-1})^\alpha\right).\end{eqnarray*}
The proof finishes when we note that $\diam(\mathcal{M})^{\alpha}\leq \gamma^{(\alpha)}_2(\mathcal{M})$ and take the infimum over admissible sequences. 
\qed
\end{proof}

We recall the following fundamental result due to Panchenko. It establishes a sub-Gaussian tail for the deviation of an heavy-tailed empirical process around its mean after a proper \emph{self-normalization} by a random quantity $\widehat V$.
\begin{theorem}[Panchenko's inequality \cite{panchenko2003}]\label{thm:panchenko}
Let $\alg$ be a finite family of measurable functions $g:\Xi\rightarrow\re$ such that $\probn g^2(\cdot)<\infty$. Let also $\{\xi_j\}_{j=1}^N$ and $\{\eta_j\}_{j=1}^N$ be both i.i.d. samples drawn from a distribution $\probn$ over $\Xi$ which are independent of each other. Define
$$
\mathsf{S}:=\sup_{g\in\alg}\sum_{j=1}^Ng(\xi_j),\quad\quad\mbox{and}\quad\quad\widehat{V}:=\esp\left\{\sup_{g\in\alg}\sum_{j=1}^N\left[g(\xi_j)-g(\eta_j)\right]^2\Bigg|\sigma(\xi_1,\ldots,\xi_N)\right\}.
$$
Then, for all $t>0$,
$$
\prob\left\{\mathsf{S}-\esp[\mathsf{S}]\ge \sqrt{\frac{2(1+t)}{N}\widehat{V}}\right\}\bigvee\prob\left\{\mathsf{S}-\esp[\mathsf{S}]\le -\sqrt{\frac{2(1+t)}{N}\widehat{V}}\right\}\le 2e^{-t}.
$$
\end{theorem}

The following result is a direct consequence of Theorem \ref{thm:panchenko} applied to the unitary class $\alg:=\{g\}$. It provides a sub-Gaussian tail for any random variable with finite $2$nd moment in terms its variance and empirical variance.
\begin{lemma}[Sub-Gaussian tail for self-normalized sums]\label{lemma:concentration:ineq:self:norm}
Suppose $\{\xi_j\}_{j=1}^N$ is i.i.d. sample of a distribution $\probn$ over $\Xi$ and denote by $\widehat\probn$ the correspondent empirical distribution. Then for any  measurable function $g:\Xi\rightarrow\re$ satisfying $\probn g(\cdot)^2<\infty$ and, for any $t>0$,
\begin{eqnarray*}
&&\prob\left\{(\widehat\probn-\probn)g(\cdot)\ge\sqrt{\frac{2(1+t)}{N}\left(\widehat\probn+\probn\right)\left[g(\cdot)-\probn g(\cdot)\right]^2}\right\}\le2e^{-t},\\
&&\prob\left\{(\widehat\probn-\probn)g(\cdot)\le-\sqrt{\frac{2(1+t)}{N}\left(\widehat\probn+\probn\right)\left[g(\cdot)-\probn g(\cdot)\right]^2}\right\}\le 2e^{-t}.
\end{eqnarray*}
\end{lemma}

Finally, we present the sub-gaussian tail of nonnegative random variables. 

\begin{lemma}[sub-Gaussian lower tail for nonnegative random variables]\label{lemma:sub-gaussian:lower:tail}
Let $\{Z_j\}_{j=1}^N$ be i.i.d. nonnegative random variables. Assume $a\in(1,2]$ and $0<\esp[Z_1^a]<\infty$. Then, for all $\epsilon>0$,
$$
\prob\left\{\frac{1}{N}\sum_{j=1}^NZ_j\le(1-\epsilon)\esp[Z_1]\right\}\le\exp\left\{-\left(\frac{a-1}{a}\right)\epsilon^{\frac{a-1}{a}}\left\{\frac{(\esp[Z_1])^a}{\esp[Z_1^a]}\right\}^{\frac{1}{a-1}}N\right\}.
$$
\end{lemma}
\begin{proof}
Let $\theta,\epsilon>0$. By the usual ``Bernstein trick'', we get
\begin{eqnarray}
\prob\left\{\frac{1}{N}\sum_{j=1}^NZ_j\le(1-\epsilon)\esp[Z_1]\right\}&\le & 
\prob\left\{\sum_{j=1}^N(\esp[Z_i]-Z_i)\ge\epsilon\esp[Z_1]N\right\}\nonumber\\
&\le &\prob\left\{e^{\theta\sum_{j=1}^N(\esp[Z_i]-Z_i)}\ge e^{\theta\epsilon\esp[Z_1]N}\right\}\nonumber\\
&\le &e^{-\theta\epsilon\esp[Z_1]N}\esp\left[e^{\theta\sum_{j=1}^N(\esp[Z_i]-Z_i)}\right]\nonumber\\
&=&e^{-\theta\epsilon\esp[Z_1]N}\esp\left[e^{\theta(\esp[Z_1]-Z_1)}\right]^N.\label{equation:sub-gaussian:lower:tail:eq1}
\end{eqnarray}

It is a simple calculus exercise to show that 
$
\forall x\ge0,e^{-x}\le1-x+\frac{x^a}{a}.
$
Applying this with $x:=\theta Z_1$, we obtain 
\begin{eqnarray*}
\esp\left[e^{\theta(\esp[Z_1]-Z_1)}\right]&\le & 
e^{\theta\esp[Z_1]}\left(1-\esp[\theta Z_1]+\frac{\esp[(\theta Z_1)^a]}{a}\right)\\
&\le &e^{\theta\esp[Z_1]}e^{-\theta\esp[Z_1]+\frac{\esp[(\theta Z_1)^a]}{a}}=e^{\frac{\esp[(\theta Z_1)^a]}{a}},
\end{eqnarray*}
where the second inequality follows from the relation $1+x\le e^x$ for all $x\in\re$. We plug this back into \eqref{equation:sub-gaussian:lower:tail:eq1} and get, for all $\theta>0$,
\begin{eqnarray}\label{equation:sub-gaussian:lower:tail:eq2}
\prob\left\{\frac{1}{N}\sum_{j=1}^NZ_j\le(1-\epsilon)\esp[Z_1]\right\}&\le &
e^{\left(-\theta\epsilon\esp[Z_1]+\theta^a\frac{\esp[Z_1^a]}{a}\right)N}.
\end{eqnarray}
Since $a\in(1,2]$, we may actually minimize the above bound over $\theta>0$. The minimum is attained at 
$
\theta_*:=\left(\frac{\epsilon\esp[Z_1]}{\esp[Z_1^a]}\right)^{\frac{1}{a-1}}.
$
To finish the proof, we plug this in \eqref{equation:sub-gaussian:lower:tail:eq2} and notice that
$$
-\theta_*\epsilon\esp[Z_1]+\theta_*^a\frac{\esp[Z_1^a]}{a}=\left(-1+\frac{1}{a}\right)\frac{(\epsilon\esp[Z_1])^{\frac{a}{a-1}}}{\esp[Z_1^a]^{\frac{1}{a-1}}},
$$
using that $1+\frac{1}{a-1}=\frac{a}{a-1}$.\qed
\end{proof}


\begin{thebibliography}{50}

\bibitem{artstein:wets} Artstein, Z. and Wets, R.J-B.: Consistency of minimizers and the SLLN for stochastic programs, Journal of Convex Analysis 2, 1-17 (1995)

\bibitem{bartlett2005} Bartlett, P., Bousquet, O. and Mendelson, S.: Local Rademacher complexities.
Ann. Statist. 33 1497–1537 (2005).

\bibitem{bartlett2006} Bartlett, P. and Mendelson, S.: Empirical minimization. Probability Theory and Related Fields
135 (3), 311–334 (2006). 

\bibitem{2012bartlett:mendelson:neeman} Barlett, P.L., Mendelson, S. and Neeman, J.: $\ell_1$-regularized linear regression: persistence and oracle inequalities, Probab. Theory Relat. Fields 154, 193--224 (2012)

\bibitem{2009bickel:ritov:tsybakov} Bickel, P.J., Ritov, Y. and Tsybakov, A.B.: Simultaneous analysis of the Lasso and Dantzig Selector,  Ann. Statist. 37(4), 1705--1732 (2009)

\bibitem{2018bellec:lecue:tsybakov}  Bellec, P.C.,  Lecu\'e, G. and  Tsybakov, A.B.: Slope meets lasso: improved oracle bounds and optimality. Ann. Statist.
46(6B), 3603--3642 (2018)

\bibitem{2007bunea:tsybakov:wegkamp}  Bunea, F., Tsybakov, A.B. and Wegkamp, M. H.: Sparsity oracle inequalities for the Lasso. Electron. J. Statist. 1, 169--194 (2007)

\bibitem{2007bunea:tsybakov:wegkamp-aggregation:gaussian}   Bunea, F., Tsybakov, A.B. and Wegkamp, M. H.: Aggregation for Gaussian regression. Ann. Statist.
35(4), 1674--1697 (2007)

\bibitem{2007bunea:tsybakov:wegkamp-density}   Bunea, F., Tsybakov, A.B. and Wegkamp, M. H.: Sparse density estimation with $\ell_1$ penalties. In: Bshouty N.H., Gentile C. (eds) Learning Theory. COLT 2007. Lecture Notes in Computer Science, vol 4539. Springer, Berlin, Heidelberg.

\bibitem{2006bunea:tsybakov:wegkamp}  Bunea, F., Tsybakov, A.B. and Wegkamp, M. H.: Aggregation and sparsity via $\ell_1$-penalized least squares. In: Lugosi G., Simon H.U. (eds) Learning Theory. COLT 2006. Lecture Notes in Computer Science, vol 4005. Springer, Berlin, Heidelberg.

\bibitem{2004bunea:tsybakov:wegkamp}   Bunea, F., Tsybakov, A.B. and Wegkamp, M. H.: Aggregation for regression learning. Preprint: https://arxiv.org/abs/math/0410214
 (2004)

\bibitem{2007candes:tao} Candes, E. and Tao, T.: The Dantzig selector: Statistical estimation when p is much larger than n. Ann. Statist. 35(6), 2313--2351 (2007)

\bibitem{dupacova:wets} Dupacov\`a, J. and Wets, R.J-B.: Asymptotic behavior of statistical estimators and of optimal solutions of stochastic optimization problems, Ann. Statist. 16(4), 1517--1549 (1988)

\bibitem{2008geer} van de Geer, S.A.: High-dimensional generalized linear models and the Lasso. Ann. Statist.
 36(2), 614--645 (2008)

\bibitem{guigues:juditsky:nemirovski} Guigues, V., Juditsky, A. and Nemirovski, A.: Non-asymptotic confidence bounds for the optimal value of a stochastic program, Optimization Methods and Software 32(5), 1033--1058 (2017)

\bibitem{2006greenshtein} Greenshtein, E.: Best subset selection, persistence in high-dimensional statistical learning and optimization under $\ell_1$ constraint.     Ann. Statist. 34(5), 2367--2386  (2006)  

\bibitem{2004greenshtein:ritov} Greenshtein, E., Ritov, Y.: Persistence in high-dimensional linear predictor selection and the virtue of overparametrization. Bernoulli 10(6), 971–988 (2004)

\bibitem{tito:bayraksan} Homem-de-Mello, T. and Bayraksan, G.: Monte Carlo sampling-based methods for stochastic optimization, Surveys in Operations Research and Management Science, 19, 56-85 (2014)

\bibitem{iusem:jofre:thompson2015} AN Iusem, A Jofr\'e, P Thompson, Incremental constraint projection methods for monotone stochastic variational inequalities, Mathematics of Operations Research 44 (1), 236-263 (2019)
 
 \bibitem{pasupathy} Kim, S., Pasupathy, R. and Henderson, S.G.: A guide to Sample Average Approximation. In: Michael Fu (ed.), Handbook of Simulation Optimization, International Series in Operations Research \& Management Science, Vol. 216, pp. 207-243. Springer, New York (2015)

\bibitem{king:rockafellar} King, A.J. and Rockafellar, R.T.: Asymptotic theory for solutions in statistical estimation and stochastic programming, Math. Oper. Res. 18, 148-162 (1993)

\bibitem{king:wets} King, A.J. and Wets, R.J-B.: Epi-consistency of convex stochastic programs, Stoch. Stoch. Rep. 34, 83-92 (1991)
\bibitem{pflug1995} Pflug, G.C.: Asymptotic stochastic programs, Math. Oper. Res. 20, 769-789 (1995)

\bibitem{2011koltchinskii} Koltchinskii, V.: Oracle inequalities in empirical risk minimization and sparse recovery problems.  Lecture Notes in Mathematics book series (LNM, volume 2033),  Ecole d'Et\'e Probabilit. Saint-Flour book sub series (LNMECOLE, volume 2033), Springer-Verlag Berlin Heidelberg, (2011)

\bibitem{2011koltchinskii:lounici:tsybakov}  Koltchinskii, V., Lounici, K. and Tsybakov, A.B.: Nuclear-norm penalization and optimal rates for noisy low-rank matrix completion.     Ann. Statist.
39(5), 2302--2329 (2011)

\bibitem{2009koltchinskii-dantzig} Koltchinskii, V.: The Dantzig selector and sparsity oracle inequalities.     Bernoulli 15(3), 799--828 (2009)

\bibitem{2009koltchinskii} Koltchinskii, V.: Sparsity in penalized empirical risk minimization. Ann. Inst. H. Poincar\'e Probab. Statist. 45(1), 7--57 (2009)

\bibitem{2009koltchinskii-entropy} Koltchinskii, V.: Sparse recovery in convex hulls via entropy penalization. Ann. Statist. 37(3), 1332--1359 (2009)

\bibitem{koltchinskii2006} Koltchinskii, V.: Local Rademacher complexities and oracle inequalities in risk minimization. Ann. Statist. 34 (6), 2593-2656 (2006).

\bibitem{2012lecue:mendelson}  Lecu\'e, G.and Mendelson, S.: General nonexact oracle inequalities for classes with subexponential envelope. Ann. Statist.
 40(2), 832--860 (2012)

\bibitem{lecue:mendelson2017} Lecu\'e, G. and Mendelson, S.: Sparse recovery under weak moment assumptions. J. Eur. Math. Soc. 19, 881--904 (2017)

\bibitem{2006leng:lin:wahba} Leng, C., Lin, Y., Wahba, G.: A note on the lasso and related procedures in model selection. Statistica Sinica 16, 1273--1284 (2006)

\bibitem{2008lounici} Lounici, K.: Sup-norm convergence rate and sign concentration property of Lasso and Dantzig estimators. Electron. J. Statist. 2, 90--102 (2008)

\bibitem{2007meinshausen:yu} Meinshausen, N. and Yu, B.: Lasso-type recovery of sparse representations for high-dimensional data. Ann. Statist.
 37(1), 246--270 (2009)

\bibitem{2007meinshausen} Meinhausen, N.: Relaxed lasso. Computational Statistics \& Data Analysis 
 52(1), 374--393 (2007).

\bibitem{2007meinshausen:buhlmann} Meinhausen, N and B\"uhlmann, P.: High-dimensional graphs and variable selection with the Lasso. Ann. Statist.
  34(3), 1436--1462 (2006)

\bibitem{oliveira2013} Oliveira, R.I.: \emph{The lower tail of random quadratic forms, with applications to ordinary least squares and restricted eigenvalue properties}, (2013), preprint at https://arxiv.org/abs/1312.2903.

\bibitem{oliveira2016} Oliveira, R.I.: \emph{The lower tail of random quadratic forms with applications to ordinary least squares}, Probab. Theory Relat. Fields 166, 1175--1194 (2016)

\bibitem{2020oliveira:thompsonI} Oliveira, R.I. and Thompson, P.: \emph{Sample average approximation with heavier tails I: non-asymptotic bounds with weak assumptions and stochastic constraints}, Mathematical Programming (2022),  https://doi.org/10.1007/s10107-022-01810-x. 

\bibitem{panchenko2003} Panchenko, D.: \emph{Symmetrization approach to concentration inequalities for empirical processes}, The Annals of Probability 31, 2068--2081 (2003)

\bibitem{pang} Pang, J-S.: \emph{Error bounds in mathematical programming}, Mathematical Programming Ser. B 79(1), 299--332 (1997)

\bibitem{pflug1999} Pflug, G.C.: Stochastic programs and statistical data, Annals of Operations Research 85, 59-78 (1999)
\bibitem{pflug2003} Pflug, G.C.: Stochastic optimization and statistical inference. In: Ruszczy\'nski, A. and Shapiro, A. (eds.) Handbooks in OR \& MS, Vol. 10, pp. 427-482. Elsevier (2003). 

\bibitem{rockafellar:urysaev2000} Rockafellar, R.T. and Urysaev, S.: Optimization of conditional value-at-risk, Journal of Risk 2(3), 493-517 (2000)

\bibitem{roemisch2003} R\"omisch, W.: Stability of Stochastic Programming Problems. In: Ruszczy\'nski, A. and Shapiro, A. (eds.) Handbooks in OR \& MS, Vol. 10, pp. 483-554. Elsevier (2003).

\bibitem{shapiro1989} Shapiro, A.: Asymptotic properties of statistical estimators in stochastic programming, Ann. Statist. 17, 841-858 (1989)

\bibitem{shapiro1991} Shapiro, A.: Asymptotic analysis of stochastic programs, Ann. Oper. Res. 30, 169-186 (1991)

\bibitem{shapiro2003} Shapiro, A.: Monte Carlo sampling methods. In: Ruszczy\'nski, A. and Shapiro, A. (eds.) Handbooks in OR \& MS, Vol. 10, pp. 353-425. Elsevier (2003). 

\bibitem{shapiro:dent:rus} Shapiro, A.. Dentcheva, D. and Ruszczynski, A.: Lectures on Stochastic Programming: Modeling and Theory. MOS-SIAM Ser. Optim., SIAM, Philadelphia, (2009).

\bibitem{talagrand1994} Talagrand, M.: Sharper bounds for Gaussian and empirical processes, Annals of Probability 22, 28-76 (1994)

\bibitem{talagrand2014} Talagrand, M.: {\em Upper and lower bounds for stochastic processes.} Springer-Verlag (2014).

\bibitem{tibshirani1996} Tibshirani, R.: \emph{Regression shrinkage and selection via the Lasso}, J. Roy. Statist. Soc. Ser. B 58, 267--288 (1996)

\bibitem{2006zhang:huang}  Zhang, C.-H. and Huang, J.: The sparsity and the bias of the lasso selection in high-dimensional linear regression. Ann. of Stat. 36(4), 1567--1594 (2008)

\bibitem{2006zhao:yu} Zhao, P. and Yu, B.: On model selection consistency of Lasso. Journal of Machine Learning Research 7, 2541--2563 (2006)

\bibitem{2009zhang} Zhang, T.: Some sharp performance bounds for least squares regression with L1 regularization.  Ann. Statist. 37(5A), 2109--2144 (2009)

\bibitem{2006zou} Zou, H.: The adaptive lasso and its oracle properties. Journal of the American Statistical Association 101(476), 1418--1429 (2006).

\end{thebibliography}
\end{document}